\documentclass[10pt]{amsart}
\usepackage{amssymb}
\usepackage{dsfont}
\usepackage{amscd}
\usepackage[mathscr]{euscript} 
\usepackage[all]{xy}
\usepackage{url}
\usepackage{stmaryrd}
\usepackage{comment}
\usepackage[retainorgcmds]{IEEEtrantools}
\usepackage{hyperref}
\title{Difference sheaves and torsors}
\author[M. CHA{\L}UPNIK]{Marcin Cha{\l}upnik$^{\dagger}$}
\thanks{$^{\dagger}$ Supported by the Narodowe Centrum Nauki  grant no. 2015/19/B/ST1/01150}
\address{$^{\dagger}$ Instytut Matematyki\\Uniwersytet Warszawski\\Warszawa\\Poland}
\email{mchal@mimuw.edu.pl}
\author[P. KOWALSKI]{Piotr Kowalski$^{\spadesuit}$}
\thanks{$^{\spadesuit}$ Supported by the Narodowe Centrum Nauki grants no. 2015/19/B/ST1/01150, 2015/19/B/ST1/01151, 2018/31/B/ST1/00357, and by T\"{u}bitak 1001 grant no. 119F397.}
\address{$^{\spadesuit}$ Instytut Matematyczny\\
Uniwersytet Wroc{\l}awski\\
Wroc{\l}aw\\
Poland}
\email{pkowa@math.uni.wroc.pl} \urladdr{http://www.math.uni.wroc.pl/\textasciitilde pkowa/ }

\DeclareMathOperator{\gl}{GL} \DeclareMathOperator{\aut}{Aut} \DeclareMathOperator{\id}{id}
 \DeclareMathOperator{\fr}{Fr} 
  
\DeclareMathOperator{\im}{im}  \DeclareMathOperator{\gal}{Gal}

 \DeclareMathOperator{\alg}{alg}

\DeclareMathOperator{\coli}{\underrightarrow{\lim}}

\DeclareMathOperator{\spec}{Spec}\DeclareMathOperator{\rat}{rat}
\DeclareMathOperator{\sep}{sep}

\DeclareMathOperator{\pic}{Pic}
\DeclareMathOperator{\modd}{Mod}
\DeclareMathOperator{\coker}{coker}
\DeclareMathOperator{\Hom}{Hom}

\DeclareMathOperator{\op}{op}\DeclareMathOperator{\as}{AS}
\DeclareMathOperator{\Ob}{Ob}

\newtheorem{theorem}{Theorem}[section]
\newtheorem{prop}[theorem]{Proposition}
\newtheorem{lemma}[theorem]{Lemma}
\newtheorem{cor}[theorem]{Corollary}

\theoremstyle{definition}
\newtheorem{definition}[theorem]{Definition}
\newtheorem{example}[theorem]{Example}
\newtheorem{remark}[theorem]{Remark}

\newtheorem{assumption}[theorem]{Assumption}

\begin{document}
\newcommand{\lili}{\underleftarrow{\lim }}
\newcommand{\coco}{\underrightarrow{\lim }}
\newcommand{\twoc}[3]{ {#1} \choose {{#2}|{#3}}}
\newcommand{\thrc}[4]{ {#1} \choose {{#2}|{#3}|{#4}}}
\newcommand{\Zz}{{\mathds{Z}}}
\newcommand{\Ff}{{\mathds{F}}}
\newcommand{\Cc}{{\mathds{C}}}
\newcommand{\Rr}{{\mathds{R}}}
\newcommand{\Nn}{{\mathds{N}}}
\newcommand{\Qq}{{\mathds{Q}}}
\newcommand{\Kk}{{\mathds{K}}}
\newcommand{\Pp}{{\mathds{P}}}
\newcommand{\ddd}{\mathrm{d}}
\newcommand{\Aa}{\mathds{A}}
\newcommand{\dlog}{\mathrm{ld}}
\newcommand{\ga}{\mathbb{G}_{\rm{a}}}
\newcommand{\gm}{\mathbb{G}_{\rm{m}}}
\newcommand{\gaf}{\widehat{\mathbb{G}}_{\rm{a}}}
\newcommand{\gmf}{\widehat{\mathbb{G}}_{\rm{m}}}
\newcommand{\ka}{{\bf k}}
\newcommand{\ot}{\otimes}
\newcommand{\si}{\mbox{$\sigma$}}
\newcommand{\ks}{\mbox{$({\bf k},\sigma)$}}
\newcommand{\kg}{\mbox{${\bf k}[G]$}}
\newcommand{\ksg}{\mbox{$({\bf k}[G],\sigma)$}}
\newcommand{\ksgs}{\mbox{${\bf k}[G,\sigma_G]$}}
\newcommand{\cks}{\mbox{$\mathrm{Mod}_{({A},\sigma_A)}$}}
\newcommand{\ckg}{\mbox{$\mathrm{Mod}_{{\bf k}[G]}$}}
\newcommand{\cksg}{\mbox{$\mathrm{Mod}_{({A}[G],\sigma_A)}$}}
\newcommand{\cksgs}{\mbox{$\mathrm{Mod}_{({A}[G],\sigma_G)}$}}
\newcommand{\catgrs}{\mbox{$\modd_{{\bf G}}^{\sigma}$}}
\newcommand{\crats}{\mbox{$\mathrm{Mod}^{\rat}_{(\mathbf{G},\sigma_{\mathbf{G}})}$}}
\newcommand{\crat}{\mbox{$\mathrm{Mod}_{\mathbf{G}}$}}
\newcommand{\cratinv}{\mbox{$\mathrm{Mod}^{\rat}_{\mathbb{G}}$}}
\newcommand{\ra}{\longrightarrow}
\newcommand{\mc}{\mathcal}
\newcommand{\mcf}{{\mathcal F}}
\newcommand{\st}{{\mathcal O}_X}
\newcommand{\sx}{\mathrm{Sh}({\mathbf{C}}(X))}
\newcommand{\psx}{\mathrm{PSh}({\mathbf{C}}(X))}
\newcommand{\sdx}{\mathrm{Sh}_{\sigma}({\mathbf{C}}(X))}
\newcommand{\dx}{\mathrm{Sh}_{\sigma}(X)}
\newcommand{\rx}{\mathrm{Sh}^{\sigma}(X)}
\newcommand{\shx}{\mathrm{Sh}_{\sigma}({\mathbf{C}}(X))}
\newcommand{\pshx}{\mathrm{PSh}_{\sigma}({\mathbf{C}}(X))}
\newcommand{\shxr}{\mathrm{Sh}^{\sigma}({\mathbf{C}}(X))}
\newcommand{\pshxr}{\mathrm{PSh}^{\sigma}({\mathbf{C}}(X))}

\begin{abstract}
We develop sheaf theory in the context of difference algebraic geometry. We introduce categories of difference sheaves and develop the appropriate cohomology theories. As specializations, we get difference Galois cohomology, difference Picard group and a good theory of difference torsors.
\end{abstract}

\maketitle

\tableofcontents


\section*{Introduction}\label{secintro}
In this paper, we develop a theory of difference sheaves. The adjective ``difference'' refers here to the situation when one turns the attention from the objects of a certain category to the pairs consisting of these objects together with a chosen endomorphism. For example, a difference ring is a ring with a fixed ring endomorphism. \emph{Difference algebra} (that is, the theory of difference rings) was initiated by Ritt and pursued further by Cohn \cite{cohn}. \emph{Difference algebraic geometry} is a global version of difference algebra and it was developed in the seminal paper of Hrushovski \cite{HrFro}. The methods of difference algebraic geometry have been recently successfully applied (usually, via the corresponding model theory) to the number of topics including algebraic dynamics (\cite{ChaHr1, ChaHr2, medsc}) and diophantine geometry (\cite{HrFro, Hr9}).

In the context of difference algebra, there is a natural abelian category of difference modules (see e.g. Chapter 3 in \cite{levin}). As far as we know, a global analogue of the notion of a difference module has not been considered yet. In this paper, we introduce difference sheaves, which is the desired global analogue of the notion of a difference module.

Our initial motivation came from a result of Anand Pillay and the second author (\cite{KP4}) saying that
difference fields of the form $(\Ff_p^{\alg},x\mapsto x^p)$ are \emph{linearly closed}, i.e. any finite-dimensional $\sigma$-module over such a difference field is isomorphic to a trivial one.
The proof of this result actually shows that for any difference field $(\ka,\sigma)$, the set of isomorphism classes of one-dimensional $\sigma$-modules,
which may be thought of as the ``difference Picard group'' of the difference field
$(\ka,\sigma)$, corresponds to the cokernel of the ``Artin-Schreier map'' $x\mapsto \sigma(x)x^{-1}$. Since this cokernel coincides with a certain group of coinvariants, the above classification has a homological flavour. These observations led us to the idea of introducing a suitable category of difference sheaves, whose cohomology would classify difference torsors (a similar result appeared in \cite{BW}, however, without an ambient abelian category).

In various difference situations, it is natural and sufficient to consider objects equipped with automorphisms instead of just endomorphisms, which  results in a much simpler theory. Thanks to the uniqueness of the inversive closure, such a situation occurs naturally in the case of the model theory of difference fields (see e.g. \cite{acfa1}).  However, the main and motivating  example of a difference scheme comes from the observation that any scheme over a finite field is naturally equipped with the absolute Frobenius endomorphism. Thus, the category of difference schemes is a natural tool for studying positive characteristic schemes. This is why we do not assume that our endomorphisms are invertible. Such a choice is a source of a good portion of technical difficulties which we encounter in our theory, but in the process of dealing with those difficulties some interesting phenomena can be observed as well.

We describe now our approach and the results of the paper in more detail.
We fix a scheme $X$  with an endomorphism $\sigma$ and  we choose  a Grothendieck topology $\mathbf{C}(X)$ on $X$. We would like to build an abelian category out
of the category of sheaves on $\mathbf{C}(X)$, which would take into account the action of $\sigma$ on sheaves.  This is why the direct image functor $\sigma_*$ should be involved in the definition as well. Hence, our difference sheaf category should not be just the category of sheaves with endomorphisms, which is considered in a greater generality in \cite[page 175]{GM}. Guided by the affine case of difference rings (see Section \ref{shmodsec}), the most natural choice is to say that a difference sheaf on $(X,\sigma)$ is a sheaf $\mathcal{F}$ together with a sheaf morphism
$\sigma_{\mathcal{F}}:\mathcal{F}\to\sigma_*({\mathcal{F}})$. On the other hand, for certain applications like those concerning vector bundles, it is more convenient to work with the sheaf $\sigma^*(\mathcal{F})$ instead of the sheaf $\sigma_*(\mathcal{F})$. However, the approach mentioned above allows one to deal with this situation as well, since an equivalent difference information is encoded
in the adjunct morphism $\widetilde{\sigma}_{\mathcal{F}}: \sigma^*(\mathcal{F})\to \mathcal{F}$. This is our initial choice: we investigate in Sections \ref{secleftcat}--\ref{secdiftor} the category
$\mathrm{Sh}_{\sigma}(\mathbf{C}(X))$ of {\em left difference sheaves of abelian groups on $\mathbf{C}(X)$} consisting
of pairs $(\mathcal{F},\sigma_{\mathcal{F}})$ as above with the obvious morphisms (see Definition \ref{defdifs}).

Establishing the basic properties of the category $\mathrm{Sh}_{\sigma}(\mathbf{C}(X))$ is not an entirely trivial task. If we consider difference sheaves of $\mathcal{O}_X$-modules in the affine case, then this problem can be quickly solved
 by observing that the category of difference modules is equivalent to the category of modules over the ring of twisted polynomials.
However, this last ring is noncommutative and the approach above cannot be applied to the non-affine case.  We bypass this difficulty by introducing the {\em diference site} $\mathbf{C}_{\sigma}(X)$ (see Section \ref{secds}) and then we show (see Theorem \ref{toposprop2}) that the category $\mathrm{Sh}_{\sigma}(\mathbf{C}(X))$ is equivalent to
the category $\mathrm{Sh}(\mathbf{C}_{\sigma}(X))$ of the usual sheaves on the site $\mathbf{C}_{\sigma}(X)$. Therefore, sites emerge naturally in our context even in the case when our chosen Grothendieck topology $\mathbf{C}(X)$ on $X$ is just the Zariski topology.
Thanks to the interpretation above, we immediately obtain that $\mathrm{Sh}_{\sigma}(\mathbf{C}(X))$  is a Grothendieck category, hence it is abelian with enough injectives and we can use the standard tools of homological algebra. Thus, we develop the theory of {\em difference sheaf cohomology} in Section \ref{secleftdc}. In particular, we introduce difference \v Cech cohomology (both for presheaves and sheaves), which is crucial later for classifying difference torsors. After establishing its basic properties, we demonstrate the adequacy of our theory in Section \ref{secdiftor} by showing that the difference cohomology groups, analogously to the classical context, classify difference bundles and some more general difference torsors as well. Then, we look more closely at the classical special cases such as the Picard group (Theorem \ref{difpicthm}) and Galois cohomology in the isotrivial context.

Quite surprisingly, the formalism discussed above is not sufficient for classifying difference torsors over arbitrary difference group schemes. More precisely, it applies to difference group schemes ``defined over constants'' (see Section \ref{igs}), which cover important classical examples (like the Picard group), but, in general, the attempts to classify difference $G$-torsors by  cohomology of left difference sheaves break down. The technical reason for this is that the sheaf $\mathcal{R}(G)$ represented by a difference group scheme $G$ is  not a left difference sheaf, which is another manifestation of the noncommutative nature of difference algebraic geometry. In fact, the structure carried by $\mathcal{R}(G)$ is somewhat dual. That is, in the case when the site $\mathbf{C}(X)$ is chosen to be the big flat site, we have a sheaf map
$$\mathcal{R}(G)\longrightarrow \sigma^*(\mathcal{R}(G)).$$
This observation prompted us to introduce the notion of a {\em right difference sheaf} as a pair $(\mathcal{F},\sigma^{\mathcal{F}})$ consisting of a sheaf $\mathcal{F}$ together with a sheaf morphism
$\sigma^{\mathcal{F}}:\mathcal{F}\to\sigma^*({\mathcal{F}})$. In Sections \ref{secrds}--\ref{rdstsec}, we develop a theory of such difference sheaves.
Despite the apparent symmetry with the category $\mathrm{Sh}_{\sigma}(\mathbf{C}(X))$, the category
$\mathrm{Sh}^{\sigma}(\mathbf{C}(X))$ of right difference sheaves is much harder to handle. Although we
eventually show that it is a Grothendieck category (by a mixture of ad hoc arguments), we are not able to develop a theory of right difference sheaves, which would be fully parallel to the theory of left difference sheaves. For example, due to the problems with comparing right difference sheaves and presheaves, we only consider the right difference \v Cech complex for sheaves, which is not fully satisfactory from the conceptual point of view. Some parts of the structure of the category $\mathrm{Sh}^{\sigma}(\mathbf{C}(X))$ are quite elusive as well. For example, $\mathrm{Sh}^{\sigma}(\mathbf{C}(X))$ is a Grothendieck category, so it has all limits. But even infinite products seem not to restrict to ordinary products of sheaves; at least, we do not see any reasonable right difference structure on the infinite product (taken in the category of sheaves) of right difference sheaves. Nevertheless, we achieve our main goal and in Section \ref{rdstsec} we classify arbitrary difference torsors by right difference cohomology groups (Theorem \ref{rdiso}) and introduce higher difference Galois cohomology (Definition \ref{hdgc} and Theorem \ref{wibdesc}).

After finishing this paper, it was pointed out to us by Ivan Toma\v{s}i\'{c} that our left difference sheaves correspond to certain objects, which were identified by him in topos-theoretic terms (see \cite[Section 22.1]{Tomasic}). It remains to be seen whether our right difference sheaves have a similar interpretation.

To summarize, in Section \ref{secconv}, we set our notation and our categorical set-up. Afterwards, the article consists of two formally independent parts. In Sections \ref{secleftcat}--\ref{secdiftor}, we investigate left difference sheaves.
By constructing the difference site, we obtain a well behaved theory which is suited for classical applications.
In Sections \ref{secrds}--\ref{rdstsec}, we turn to the theory of right difference sheaves, which allows a description of difference torsors in full generality.

We would like to thank Ivan Toma\v{s}i\'{c} for telling us about his work (mentioned in the paragraph above) and Michael Wibmer for pointing out to us some inaccuracies in Section \ref{left1}.

\section{Notation and conventions}\label{secconv}
For convenience of the reader, in this subsection we fix some of our notation and make one global assumption.
\begin{itemize}
\item All the adjunction bijections will be denoted by $f\mapsto \widetilde{f}$ and, usually, $\widetilde{f}$ will denote the left adjunct of $f$.

\item We denote the contravariant functor represented by the object $X$ in the category $\mathcal{C}$ by:
$$\mathcal{R}(X):C^{\op}\ra \mathbf{Set}.$$

\item We take the definition of a \emph{site} from \cite{tamme}, so, a site $\mathbf{S}$ is a pair $(\mathcal{C},J)$ such that $\mathcal{C}$ is a category and $J$ is a Grothendieck topology on $\mathcal{C}$, that is: a set of \emph{coverings} that is families of morphisms $(U_i\to U)_{i\in I}$ in $\mathcal{C}$ satisfying the conditions (T1), (T2), (T3) from \cite[Definition I.1.2.1]{tamme}. We also take the definition of a\emph{ morphism of sites} (or of topologies) from \cite[Definition I.1.2.2]{tamme}, although we choose the directions of arrows opposite to the one from \cite[Definition I.1.2.2]{tamme}.


\end{itemize}
In order to resolve the known set theoretic issues concerning sheaf categories on sites, we restrict all our constructions to a Grothendieck universe \cite{sga41}. Thus all sites we consider are small.
A reader not wishing to accept the existence of strongly inaccessible cardinals (which is necessary to make the Grothendieck approach working) is advised to consult e. g.
 \cite[\href{https://stacks.math.columbia.edu/tag/020M}{Section 020M}]{stacks}, where a simpler procedure, sufficient for the practical purposes, is described.

Let $\mathbf{S}$ be a site. The collection of sheaves on $\mathbf{S}$ forms a category, which we denote by $\mathrm{Sh}({\mathbf{S}})$. Similarly for the presheaves on $\mathbf{S}$ and the category $\mathrm{PSh}({\mathbf{S}})$.
Throughout the paper we make the following assumption.
\begin{assumption}\label{mainass}
We fix a scheme $X$, a scheme morphism $\sigma: X\to X$, and a ``site on $X$'', that is a site $\mathbf{C}(X):=(\mathcal{C},J)$ such that the following holds.
\begin{enumerate}
\item The category $\mathcal{C}$ is a subcategory of $\mathbf{Sch}_X$.

\item The category $\mathcal{C}$ is closed under any base extensions in $\mathbf{Sch}_X$ and all finite inverse limits exist in $\mathcal{C}$.

\item The Grothendieck topology $J$ on $\mathcal{C}$ is \emph{subcanonical}, that is all representable presheaves on $\mathcal{C}$ are sheaves with respect to $J$.

\item The fiber product functor (well-defined thanks to item $(2)$ above):
$$\sigma^{-1}:\mathcal{C}\ra \mathcal{C}$$
induces a morphism of sites, which we denote by the same letter as the fixed self-morphism of $X$:
$$\sigma:\mathbf{C}(X)\ra \mathbf{C}(X).$$
\end{enumerate}
\end{assumption}
By Assumption \ref{mainass}(4), we have the direct image and inverse image functors induced by $\sigma$ on the categories $\mathrm{Sh}({\mathbf{C}}(X))$ and  $\mathrm{PSh}({\mathbf{C}}(X))$, which we denote as follows:
$$\sigma_*,\sigma^*:\mathrm{PSh}({\mathbf{C}}(X))\ra \mathrm{PSh}({\mathbf{C}}(X));\ \ \ \ \
\sigma_*,\sigma^*:\mathrm{Sh}({\mathbf{C}}(X))\ra \mathrm{Sh}({\mathbf{C}}(X)).$$
The functor $\sigma_*$ is right adjoint to $\sigma^*$, so $\sigma_*$ commutes with inverse limits and $\sigma^*$ commutes with direct limits.
We would like to emphasize that, thanks to the second clause in Assumption \ref{mainass}(3), the functor $\sigma^*:\mathrm{Sh}({\mathbf{C}}(X))\to \mathrm{Sh}({\mathbf{C}}(X))$ is exact (see \cite[Proposition II.2.6(a)]{milne1etale} and the comment above \cite[Remark II.3.1]{milne1etale}). Since $\sigma^*$ commutes with kernels and finite products (which are the same as finite coproducts in any abelian category), then, by the standard arguments, $\sigma^*$ commutes with finite inverse limits.

To ease the notation, we sometimes identify an object in $\mathcal{C}$, which is a morphism, with its domain. We will mainly consider the Zariski site ${\bf Z}(X)$, the \'etale site ${\bf E}(X)$, and the big flat site ${\bf F}(X)$
as in \cite{milne1etale}, in particular all the morphisms of schemes considered here are locally of finite type. Each of these three sites satisfies Assumption \ref{mainass} with respect to any (fixed) $\sigma$.

For convenience, we introduce the following ``pullback diagram notation''. For any scheme morphism $u:U\to X$ (usually coming from the chosen $\mathbf{C}(X)$), and any morphism $\tau:X\to X$ (usually of the form $\tau=\sigma^n$ for some $n>0$):
\begin{equation*}
\xymatrix{{}^{\tau}U \ar[d]^{{}^{\tau}u} \ar[rr]^{\tau_U} & & U \ar[d]^{u}\\
X \ar[rr]^{\tau} & & X,}
\end{equation*}
where we have:
$${}^{\tau}U:=(X,\tau)\times_XU.$$

\section{Left difference sheaves}\label{secleftcat}
In this section, we introduce the first version of the category of difference sheaves. We will call the objects of this category \emph{left} difference sheaves. As pointed out in Introduction, this category has good properties and an elegant description, but it does not cover all the natural examples, hence the need for the category of \emph{right} difference sheaves, which will be introduced later in Section \ref{secrds}.

We recall that we are working under Assumption \ref{mainass}.

\subsection{Definitions and examples}
We define below the main object of this part of the paper.
\begin{definition}\label{defdifs}
\begin{enumerate}
\item We call a (pre)sheaf $\mathcal{F}$ of abelian groups on $\mathbf{C}(X)$ together with a (pre)sheaf morphism
$$\sigma_{\mathcal{F}}: \mathcal{F} \ra \sigma_*(\mathcal{F}),$$
a {\it left difference (pre)sheaf of abelian groups on $\mathbf{C}(X)$}.

\item  A morphism between left difference
(pre)sheaves $(\mathcal{F},\sigma_{\mathcal{F}})$ and
	$(\mathcal{G},\sigma_{\mathcal{G}})$ is a (pre)sheaf morphism $\alpha:\mathcal{F}
	\to \mathcal{G}$ such that
$$\sigma_{\mathcal{G}}\circ\alpha=\sigma_{*}(\alpha)\circ \sigma_{\mathcal{F}}.$$

\item The difference (pre)sheaves of abelian groups on $X$ with their morphisms form a category, which we denote $\shx$
	(resp. $\pshx$).
\end{enumerate}
\end{definition}
For a left difference sheaf $(\mathcal{F},\si_{\mcf})$, it will be often more convenient to work with the morphism
$$\widetilde{\sigma}_{\mathcal{F}}: \si^*(\mathcal{F}) \ra \mathcal{F},$$
which is adjoint to the morphism $\si_{\mathcal{F}}$. When we just say ``difference sheaf'', we always mean ``left difference sheaf''.
\begin{example}\label{moexam}
The most obvious example is $(\mathcal{O}_X,\sigma^{\sharp})$, since the scheme morphism $\sigma:X\to X$ is actually a pair $(\sigma,\sigma^{\sharp})$, where
$$\sigma^{\sharp}:\mathcal{O}_X\ra \sigma_*(\mathcal{O}_X).$$
\end{example}
\begin{remark}\label{leftoa}
Let us denote the standard sheafification (resp. forgetful) functor by $a$ (resp. by $o$). We note here that we still have the forgetful and the sheafification functors (this will be used in the proof of Proposition \ref{quickss}) in the left difference context, that is:
\begin{enumerate}
\item if $(\mathcal{F},\sigma_F)$ is a left difference presheaf, then $a(\mathcal{F})$ has a natural left difference sheaf structure;

\item if $(\mathcal{F},\sigma_F)$ is a left difference sheaf, then $o(\mathcal{F})$ has a natural left difference presheaf structure;

\item the difference sheafification functor is left adjoint to the difference forgetful functor.
\end{enumerate}
\end{remark}
\begin{proof}
For item $(1)$, we use the natural morphism:
$$(a\circ \sigma_*)(\mathcal{F}) \ra (\sigma_*\circ a)(\mathcal{F}),$$
and item $(2)$ is obvious, since $\sigma_*$ commutes with the forgetful functor $o$.

Regarding item $(3)$, it is easy to see that the standard adjunction bijection between the sheafification and the forgetful functors preserve the difference morphisms in the difference case, hence it gives an adjunction of the corresponding difference functors.
\end{proof}
We will discuss several examples starting with a straightforward one.
\begin{example}\label{constex}
Let $A$ be an abelian group and let $A_X$ be the constant sheaf. Since $\si^*(A_X) =A_X$ in any topology, the map adjoint to the identity makes $A_X$ into a left difference sheaf.

Moreover, since $A\mapsto A_X$ is a functor (left adjoint to the global section functor), for any group endomorphism $f:A\to A$, we obtain the morphism of sheaves:
 $$f_A:\sigma^*(A_X)=A _X\ra A_X,$$
and $(A_X,f_A)$ becomes a left difference sheaf. This construction will be used in Section \ref{secldifcoh}.
\end{example}
\subsubsection{Sheaves of modules}\label{shmodsec}
We consider now more complicated examples coming from difference modules.
\begin{definition}
By a  {\it left difference sheaf of $\st$-modules}, we mean a sheaf of $\st$-modules (in the Zariski topology)
$\mathcal{F}$ together with a morphism
$$\sigma_{\mathcal{F}}:\mathcal{F}\ra \sigma_*(\mathcal{F})$$
 of  $\st$-modules.
\end{definition}
In the above definition, $\sigma_*(\mathcal{F})$ (the direct image in the category of sheaves of abelian groups) has a natural structure of an $\st$-module making it the direct image in the category of $\st$-modules.
However, one should remember that when we want to take the adjoint point of view, the map
$\sigma_{\mathcal{F}}:\mathcal{F}\to \sigma_*(\mathcal{F})$ of $\st$-modules corresponds
to the map of $\st$-modules $\widetilde{\sigma}_{\mathcal{F}}:\sigma^{\diamond}(\mathcal{F})\to \mathcal{F}$, where $$\sigma^{\diamond}(\mathcal{F}):=\sigma^*(\mathcal{F})\ot_{\st}\st$$
is the inverse image in the category of $\st$-modules.

We introduce here a (rather obvious) notion of a tensor product of difference sheaves of $\st$-modules, which will be used in Section \ref{secdbdpg}.
Since the functor $\si^{\diamond}$ commutes with the tensor product (on sheaves of $\st$-modules), there is a natural difference $\st$-module structure on the tensor product given by the following composition:
\begin{equation*}
\xymatrix{\si^{\diamond}\left(\mathcal{F}\otimes \mathcal{G}\right) \ar[rr]^{\cong\ \ \ \ } & & \si^{\diamond}(\mathcal{F})\otimes \si^{\diamond}(\mathcal{G})
 \ar[rr]^{\ \ \ \ \ \widetilde{\sigma}_{\mathcal{F}}\otimes \widetilde{\sigma}_{\mathcal{G}}} & & \mathcal{F}\otimes  \mathcal{G},}
\end{equation*}
and for any left difference sheaves of $\st$-modules $(\mathcal{F},\sigma_{\mathcal{F}})$ we have:
$$(\mathcal{F},\sigma_{\mathcal{F}})\otimes (\mathcal{O}_X,s^{\sharp})\cong (\mathcal{F},\sigma_{\mathcal{F}}).$$
We assume now---and always in the context when we talk about sheaves of $\mathcal{O}_X$-modules---that the topology $J$ in our working site $\mathbf{C}(X)$ is contained in the topology of the big flat site $\mathbf{F}(X)$. It is well known (see e.g. \cite[Example II.1.2.(d)]{milne1etale} and \cite[Corollary II.1.6]{milne1etale}) that
 any quasi-coherent $\st$-module $\mathcal{F}$  can be extended to a sheaf
 $\mathcal{F}_{\mathbf{C}}$ of abelian groups on $\mathbf{C}(X)$. Explicitly, for
 $(u: U\to X)\in \mathrm{Ob}(\mathcal{C})$, we put:
 \[\mathcal{F}_{\mathbf{C}}(u):=u^{\diamond}(\mathcal{F})(U).\]
For a left difference quasicoherent sheaf of $\st$-modules $(\mathcal{F},\si_{\mathcal{F}})$, we would like to equip $\mathcal{F}_{\mathbf{C}}$ with a structure of a left difference sheaf of abelian groups. First, we functorially extend the morphism $\widetilde{\sigma}_{\mathcal{F}}:\si^{\diamond}(\mathcal{F})\to \mathcal{F}$ of sheaves on the Zariski site
 	to a morphism
 	$(\si^{\diamond}(\mathcal{F}))_{\mathbf{C}}\to \mathcal{F}_{\mathbf{C}}$ of sheaves on the site $\mathbf{C}(X)$. It suffices to construct a morphism
 $$\sigma^*(\mathcal{F}_{\mathbf{C}})\ra (\si^{\diamond}(\mathcal{F}))_{\mathbf{C}}$$
 or, equivalently, a morphism
$$\mathcal{F}_{\mathbf{C}}\ra \si_*\left((\si^{\diamond}\mathcal{F})_{\mathbf{C}}\right).$$
To this end, let us consider any $(u: U\to X)\in \mathrm{Ob}(\mathcal{C})$. We have:
\begin{IEEEeqnarray*}{rCl}
\si_*((\si^{\diamond}(\mathcal{F})_{\mathbf{C}})(u) & = & (\si^{\diamond}(\mathcal{F}))_{\mathbf{C}}({}^{\sigma}u) \\
 & = & 	\left(({}^{\sigma}u)^{\diamond}\circ \si^{\diamond}\right)(\mathcal{F})({}^{\sigma}U)\\
 & = & (\sigma_U^{\diamond}\circ u^{\diamond})(\mathcal{F})({}^{\sigma}U) \\
& = & \left((\sigma_U)_*\circ\sigma_U^{\diamond}\circ u^{\diamond}\right)(\mathcal{F})(U),
\end{IEEEeqnarray*}
where the notations ${}^{\sigma}u,{}^{\sigma}U$ were introduced at the end of Section \ref{secintro}. Since we have:
 \[\mathcal{F}_{\mathbf{C}}(u)=u^{\diamond}(\mathcal{F})(U),\]
the required morphism comes from the unit morphism
 $$\id_U\longrightarrow (\sigma_U)_*\circ\sigma_U^{\diamond}$$
 in the category of $\mathcal{O}_U$-modules after forgetting the module structure.

 We can obtain now many examples of difference sheaves on affine schemes
 by applying the construction of the associated sheaf in the difference context.
 We recall (see for example \cite{cohn}) that for a commutative ring with endomorphism $(R,s)$ a \emph{difference $R$-module} is  an $R$-module $M$ together with an $R$-linear $s_M: M\to M^{(1)}$, where $M^{(1)}$
 	stands for $M$ with the $R$-action twisted by $s$ (such a pair $(M,s_M)$ was called a {\em right} difference $R$-module in \cite{CK2}, which does not quite fit to the terminology here, however, see Section  \ref{etaleexam}).
 	If we set $X:=\spec(R)$ and $\sigma:=\spec(s)$, then, by \cite[Prop. II.5.2(d)]{ha}, we have an isomorphism of
 	$\mathcal{O}_{X}$-modules:
 	$$\sigma_*(\widetilde{M})\cong \widetilde{M^{(1)}},$$
 	and we get a left difference structure on $\widetilde{M}$. By \cite[Prop. II.5.2(b)]{ha}, the functor which associates a left difference sheaf of modules to a difference module commutes with the tensor product operation, which will be used in Section \ref{secdpaffine}.
\subsubsection{Isotrivial group schemes}\label{igs}
We discuss now representable difference sheaves.
We say that a commutative  group scheme $G$ over $X$  is \emph{defined over constants of $\sigma$}, when
 there is a scheme map $t:X\to F$ such that $\sigma\circ t=t$ and
 $$G\cong X\times_{F}G_F ,$$
where $G_F$ is a group scheme over $F$.
Then we have
\begin{lemma}\label{defoverc}
If $G$ is defined over constants, then ${}^{\sigma}G\cong G$.
\end{lemma}
\begin{proof}
We show that ${}^{\sigma}G\cong G$ by the following easy computation:
\begin{IEEEeqnarray*}{rCl}
{}^{\sigma}G & = & (X,\sigma)\times_X G \\
 &\cong & (X,\sigma)\times_X \left((X,t)\times_F G_F\right)\\
 &\cong & (X,t\circ \sigma)\times_F G_F\\
& = & (X,t)\times_F G_F\\
& = & G. \ \ \ \ \ \ \ \ \ \ \ \ \ \ \ \ \ \ \ \ \ \ \ \ \ \ \ \ \ \ \ \ \ \ \ \ \ \ \ \ \ \ \ \ \ \ \ \ \ \ \ \ \ \ \ \ \ \ \ \ \ \ \ \ \ \ \ \qedhere
\end{IEEEeqnarray*}
\end{proof}
Let  $G$ be a commutative group scheme over $X$ which is defined over constants. We recall from Section \ref{secintro} that $\mathcal{R}(G)$ denotes the representable sheaf of abelian groups given by $G$. Then, by \cite[Remark II.3.1.(e)]{milne1etale}, we have a natural map
$$\phi_G:\sigma^*(\mathcal{R}(G))\ra \mathcal{R}({}^{\sigma}G)\cong \mathcal{R}(G)$$
giving $\mathcal{R}(G)$ a left difference sheaf structure.

\subsubsection{Zariski sheaves on $\spec(\ka)$}\label{zarexam}
In this example and in the next one, the whole category of difference sheaves can be easily described.

Let $\ka$ be a field with an endomorphism $s$, $X=\spec(\ka)$ and $\si=\spec(s)$. We consider the Zariski topology on $X$, that is $\mathbf{C}(X)=\mathbf{Z}(X)$. Then the  category ${\mathrm{Sh}_{\sigma}({\mathbf{C}}(X))}$ does not depend neither on $s$ nor on $\ka$, and it is just equivalent with the category of commutative groups with endomorphisms, which is the same as the category of $\Zz[x]$-modules. We often say ``$\Zz[\sigma]$-module'' instead of ``$\Zz[x]$-module'' in this context.

\subsubsection{\'{E}tale sheaves on $\spec(\ka)$}\label{etaleexam}
Let us still take $X$ and $\sigma$ as in Section \ref{zarexam}, however, instead of the Zariski topology, we consider the \'etale topology on $X$ (that is $\mathbf{C}(X)=\mathbf{E}(X)$), so the situation becomes more interesting.
Firstly, by \cite[Theorem II.1.9]{milne1etale} there is an equivalence of abelian categories
$$\Psi: \mathrm{Sh}({\bf E}(X))\simeq \gal(\ka)-\mathrm{Mod},\ \ \ \ \Psi(\mathcal{F})=\coli_{K/\ka}\mathcal{F}(\spec(K)),$$
where $K/\ka$ is a finite separable field extension (in a fixed separable closure of $\ka$) and $ \gal(\ka)-\mathrm{Mod}$ is the category of continuous left discrete modules over $\gal(\ka):=\gal(\ka^{\sep}/\ka)$. Let $\bar{s}: \ka^{\sep}\to \ka^{\sep}$ be a field endomorphism extending $s$, and let $\tau\in \gal(\ka)$. Since we have:
 $$\tau(\bar{s}(\ka^{\sep}))\subseteq (s(\ka))^{\sep}=\bar{s}(\ka^{\sep}),$$
 we can define the following ``restriction along $s$'' continuous group homomorphism:
 $$\phi: \gal(\ka)\ra \gal(\ka),\ \ \ \ \phi(\tau)=(\bar{s})^{-1}\circ \tau \circ \bar{s},$$
 which does not depend on the choice of $\bar{s}$. By \cite[Remark II.3.1.(e)]{milne1etale}, for any $\mathcal{F}\in \mathrm{Sh}({\bf E}(X))$ we get:
 $$\Psi(\si^*(\mathcal{F}))\cong 	\Psi(\mathcal{F})^{(1)},$$
 where for a continuous $\gal(\ka)$-module $M$, $M^{(1)}$ is the $\gal(\ka)$-module with the $\gal(\ka)$-structure twisted by $\phi$. Thus we see that
 in this case, the category $\shx$ is equivalent to the category of (the continuous versions of) left difference $(\gal(\ka),\phi)$-modules in the sense of \cite{CK2}, therefore in this case (unlike in the situation considered in Section \ref{shmodsec}) the terminology choices made in this paper coincide with those from \cite{CK2}.

In the case when $s$ is an automorphism (or, more generally, when the field extension $s(\ka)\subseteq \ka$ is algebraic), we can choose $\bar{s}$ to be an automorphism of $\ka^{\sep}$, and then the formula defining $\phi$ comes from an inner automorphism of the group $\aut(\ka^{\sep})$.


\subsection{Difference site}\label{secds}
In order to quickly obtain the standard properties of the categories
$\shx$ and $\pshx$, we shall interpret them as the categories of
abelian (pre)sheaves on a certain site $\mathbf{C}_{\sigma}(X)=(\mathcal{C}_{\sigma},J_{\sigma}$) (recall that our original site $\mathbf{C}(X)$ is of the form $(\mathcal{C},J)$, see Assumption \ref{mainass}), which we call the {\it difference site}.
We describe the underlying category $\mathcal{C}_{\sigma}$ first.

The objects do not change, that is
 $\Ob(\mathcal{C}_{\sigma}):=\Ob(\mathcal{C})$. To describe the morphisms in
the category $\mathcal{C}_{\sigma}$, let us take two objects $u:U\to X, v: V\to X$. We define:
$${\rm Hom}_{\mathcal{C}_{\sigma}}(u,v):=\{(f,n)\in {\rm Hom}_{\mathrm{Sch}}(U,V)\times \Nn\ |\ v\circ f=\si^n\circ u\}.$$
The composition is defined as follows:
$$\circ:{\rm Hom}_{\mathcal{C}_{\sigma}}(u,v)\times {\rm Hom}_{\mathcal{C}_{\sigma}}(t,u)\ra {\rm Hom}_{\mathcal{C}_{\sigma}}(t,v),$$
$$(f,n)\circ (g,m):=(f\circ g,n+m).$$
It is easy to see that the above composition map is well-defined, associative, and we clearly have:
$$\id_u^{\mathcal{C}_{\sigma}}=\left(\id_u^{\mathcal{C}},0\right),$$
hence $\mathcal{C}_{\sigma}$ is a category.
\begin{remark}\label{diftoprem}
We comment here on several obvious properties of the category $\mathcal{C}_{\sigma}$, which was defined above.
\begin{enumerate}

\item The category $\mathcal{C}$ can be considered as a subcategory of $\mathcal{C}_{\sigma}$, since a scheme morphism
$f:U\to V$ belongs to $\mathcal{C}$ if and only if $(f,0)$ is a morphism in $\mathcal{C}_{\sigma}$. Formally, we have a  faithful functor
$$i:\mathcal{C}\ra \mathcal{C}_{\sigma},\ \ \ \ i(f)=(f,0),$$
which is the identity on $\mathrm{Ob}(\mathcal{C})$. (Although it is obvious from the formal point of view, we still would like to point out here that we never identify $(f,n)$ with $(g,m)$ for $n\neq m$, hence $\mathcal{C}_{\sigma}\neq \mathcal{C}$, even in the case when $\si=\id_X$.)

\item
For each $V\in \Ob(\mathcal{C})$,  we have the following scheme morphism (see the end of Section \ref{secconv}):
$$\sigma_V:{}^{\sigma}V\to V,$$
 hence  the pair  $(\sigma_V,1)$ is a morphism in $\mathcal{C}_{\sigma}$, which will be abbreviated just to $\sigma_V$. Let us now set
 $$\sigma^n_V:= \sigma_{{}^{\sigma^{n-1}}V}\circ\ldots\circ\sigma_{{}^{\sigma}V}\circ \sigma_V$$
 for $n\in\Nn$. We point out that we slightly abuse notation here, since the domain $\sigma^n_V$ is only canonically isomorphic to the pullback along
 $\sigma^n$. Thanks to this choice, we have the equality
 \[
 \sigma^m_{{}^{\sigma^n}V}\circ\sigma^n_V=\sigma^{n+m}_V.
 \]
 Then, by the universal property of pullback,
 any morphism in $\mathcal{C}_{\sigma}$, given by $(f,n): U\to V$, uniquely factorizes as:
$$(f,n)=(\sigma^n_V,n)\circ (f',0)=\sigma^n_V\circ f'$$
for a certain $f':U\to {}^{\sigma^n}V$, which is a morphism in $\mathcal{C}$. Therefore the morphisms in $\mathcal{C}_{\sigma}$ are generated by the morphisms from $\mathcal{C}$ and the morphisms of the form $\sigma^n_V$.
This may be thought of as an analogue of the fact
that the category of difference modules is equivalent to the category of modules over the ring of twisted polynomials, which is generated by the monomials $t^i r$. In order to extend this analogy, let us see how the above decomposition behaves with respect to the composition in $\mathcal{C}_{\sigma}$.
In order to factorize the composite:
\begin{equation*}
\xymatrix{U \ar[r]^{(f,n)\ } & V \ar[r]^{(g,m)\ } & W,}
\end{equation*}
we need to interchange $\sigma^n_V$ and the corresponding $g'$. To achieve this, we note the following identity in $\mathcal{C}_{\sigma}$ which will be crucial in the proof of Proposition \ref{toposprop}:
\[
g\circ \sigma^n_V=\sigma^n_W\circ\sigma^n(g).
\]
Applying this identity to the factorization above, we obtain the formula:
\[
\sigma^m_W\circ g'\circ \sigma^n_V\circ f'=\sigma^{n+m}_W\circ\sigma^n(g')\circ f',
\]
which is analogous to the formula for multiplying twisted polynomials.

\item At last, we would like to note that
the pullbacks of morphisms from $\mathcal{C}_{\sigma}$ along morphisms from $\mathcal{C}$ exist (by Assumption \ref{mainass}(2)) in the category $\mathcal{C}_{\sigma}$.
\end{enumerate}
\end{remark}
Before the next result, let us mention that to consider the notion of a presheaf, we do not need any choice of a Grothendieck topology on a category.
\begin{prop}\label{toposprop}
There is an equivalences of categories
\[\Psi:\mathrm{PSh}(\mathcal{C}_{\sigma})\ra \pshx.\]
\end{prop}
\begin{proof} Let $\mathcal{F}\in\mathrm{PSh}(\mathcal{C}_{\sigma})$. For $U\in \mathrm{Ob}(\mathcal{C})$,
we put $\Psi(\mathcal{F})(U):=\mathcal{F}(U)$. Obviously, $\Psi(\mathcal{F})\in \psx$ and in order to define a morphism of presheaves on $\mathcal{C}$
$$\sigma_{\Psi(\mathcal{F})}:\Psi(\mathcal{F})\ra \sigma_*(\Psi(\mathcal{F})),$$
we need a natural family of maps
$$\mc{F}(U)\ra \mc{F} \left({}^{\sigma}U\right)$$
for each $U\in \Ob(\mathcal{C})$. However, the fiber product structure map $\sigma_U:{}^{\sigma}U\to U$ is a morphism in the category $\mathcal{C}_{\sigma}$ (see Remark \ref{diftoprem}(2), formally the pair $(\sigma_U,1)$ is such a morphism) and it induces a morphism $\mc{F}(U)\to \mc{F}({}^{\sigma}U)$. Hence, we found a natural difference
presheaf structure on $\Psi(\mathcal{F})$, since the required naturality condition is that for a morphism $f:U\to V$ in $\mathcal{C}$, we have:
\[\mathcal{F}(\sigma_U)\circ\mathcal{ F}(f)=\mathcal{F}(\sigma(f))\circ\mathcal{ F}(\sigma_V),\]
which follows from the application of $\mathcal{F}$ to the identity
\[f\circ\sigma_U=\sigma_V\circ \sigma(f)\]
established in Remark \ref{diftoprem}(2).

On the other hand, for $\mathcal{G}\in\pshx$ we form $\Phi(\mathcal{G})\in \mathrm{PSh}(\mathcal{C}_{\sigma})$
by extending $\mathcal{G}$ to the morphisms in the category $\mathcal{C}_{\sigma}$ using the following formula:
\[
\Phi(\mathcal{G})(\sigma^n_U\circ f):=\mathcal{G}(f)\circ \sigma^n_{G}.
\]
The fact that we obtain a presheaf follows from the commutativity relation from Remark \ref{diftoprem}(2) again. It is immediate to check that $\Phi$ is a quasi-inverse to $\Psi$, hence $\Psi$ is an equivalence of categories.
\end{proof}
\noindent
We would like to make now the category $\mathcal{C}_{\sigma}$ into a site, that is, to specify a Grothendieck topology $J_{\sigma}$ on $\mathcal{C}_{\sigma}$. We define $J_{\sigma}:=J$, that is a family of $\mathcal{C}_{\sigma}$-morphisms $\{f_i\}_i$ is in $J_{\sigma}$, if it is in $J$ in the sense of Remark \ref{diftoprem}(1) (in particular, all the $f_i's$ are morphism in the category $\mathcal{C}$). We define $\mathbf{C}_{\sigma}(X)$ as $(\mathcal{C}_{\sigma},J_{\sigma})$ and we have the following.
\begin{prop}\label{sistopos}
The pair $\mathbf{C}_{\sigma}(X)$ is a site and the functor $i$ from Remark \ref{diftoprem}(1) induces a morphism of sites:
$$i:\mathbf{C}_{\sigma}(X)\ra \mathbf{C}(X).$$
\end{prop}
\begin{proof}
We will check the three properties defining the notion of a site from Section \ref{secconv}. The first property follows immediately from the corresponding property of $\mathbf{C}(X)$. The second follows from the fact that the only isomorphisms in $\mathcal{C}_{\sigma}$ are those coming from $\mathcal{C}$.
Since the pullbacks along the morphisms $(f,0)$ exist in $\mathcal{C}_{\sigma}$ (Remark \ref{diftoprem}(3)) and are preserved by the functor $i$, the third property for $\mathbf{C}_{\sigma}(X)$ follows from the corresponding property for $\mathbf{C}(X)$.

The functor $i$ is a morphism of sites, since it takes coverings in $J$ to coverings in $J_{\sigma}$ and it preserves fiber products.
\end{proof}
Using Proposition \ref{toposprop} and Proposition \ref{sistopos}, we immediately obtain the following.
\begin{theorem}\label{toposprop2}
There is an equivalences of categories
\[\Psi:\mathrm{Sh}(\mathbf{C}_{\sigma}(X))\ra \shx.\]
\end{theorem}
\begin{proof}
It is clear that for the morphism of sites $i$ from Proposition \ref{sistopos}, the corresponding direct image functor
$$i_*: \shx\to \mathrm{Sh}(X)$$
is the forgetful functor (forgetting the difference structure). Hence, this forgetful functor takes sheaves to sheaves, which implies that the equivalence of the presheaf categories $\Psi$ from Proposition \ref{toposprop} takes sheaves to sheaves, so it gives an equivalence of categories we were looking for.
\end{proof}
The interpretation above allows us to obtain all the standard properties of the categories $\pshx$ and $\shx$.
\begin{cor}
The categories $\pshx,\shx$ are Grothendieck categories. In particular, they are abelian, they have all   limits
and colimits, and they have enough injectives.
\end{cor}
Since the forgetful functor is of the form $i_*$ (see the proof of Theorem \ref{toposprop2}), some of its properties follow directly from sites generalities.  However, the situation is quite special and this functor enjoys some extra properties, which are listed below.
\begin{prop}\label{improve}
The  functor
$$i_*:\mathrm{Sh}(\mathbf{C}_{\sigma}(X))\ra \mathrm{Sh}(\mathbf{C}(X))$$
has the following  properties.
\begin{enumerate}
\item A difference presheaf $\mathcal{F}$ is a sheaf if and only if $i_*(\mathcal{F})$ is a sheaf (we use here the presheaf version of the functor $i_*$).

\item The  functor $i_*$ has a left adjoint functor $i^*$  and a right adjoint functor $i^{!}$.

\item The functor $i_*$  preserves all   limits and colimits, in particular it is exact.


\item A sequence of difference sheaves
$$0\ra \mathcal{F}\ra \mathcal{G}\ra\mathcal{H}\ra 0$$
is exact if and only if the sequence of sheaves
$$0\ra i_*(\mathcal{F})\ra i_*(\mathcal{G})\ra i_*(\mathcal{H})\ra 0$$
is exact.

\item The functor $i^*$ is exact.

\item The functors $i_*$, $i^!$ preserve injectives. Moreover, any difference sheaf embeds into
$i^!(\mcf)$ for some injective sheaf $\mcf$.
\end{enumerate}
Moreover, the facts analogous to items $(2)$--$(6)$ also hold for the category of difference presheaves instead of difference sheaves.
\end{prop}
\begin{proof}
Item (1) immediately follows from the fact that the families of coverings in the sites
$\mathbf{C}(X)$ and $\mathbf{C}_{\sigma}(X)$ are the same.

We prove items (2)--(5) for sheaves, the proofs for presheaves are similar.
For item (2),  we remark that the existence of $i^*$ follows from site generalities and also that
the existence of $i^!$ may be derived from the Special Adjoint Functor Theorem (see \cite[Section V.8]{maclane}). However, we will provide a simple explicit construction of both these functors utilizing a general difference idea, since this construction will be useful later in the proof of Proposition \ref{cechder}. For $\mc{F}\in \mathrm{Sh}(X)$ we define:
	\[
	i^{!}(\mc{F}):=\prod_{j=0}^{\infty} \sigma^j_{*}(\mathcal{F})
	\]
	with the map
	\[\sigma_{i^{!}(\mc{F})}:i^{!}(\mc{F})=\prod_{j=0}^{\infty} \sigma^j_{*}(\mathcal{F})\ra
	\sigma_{*}(i^{!}(\mc{F}))=\prod_{j=1}^{\infty} \sigma^j_{*}(\mathcal{F})\]
	being the obvious projection onto the subproduct (we use here the fact that the $\sigma_*$ commutes with products, see Section \ref{secconv}). Let
		$$\alpha=\prod_{j=0}^{\infty}\alpha_j:\mathcal F\ra i^{\!}(\mathcal{G})$$
		be a morphism of difference sheaves. Then the condition $\sigma_{i^{\!}(\mathcal{G})}
		\circ \alpha=\sigma_{*}(\alpha)\circ \sigma_{\mathcal{F}}$ is equivalent to the condition
		$\alpha_{j+1}=\sigma_{*}(\alpha_j)\circ \sigma_{\mathcal{F}}$ for all $j\geqslant 0$, which means
		that:
$$\alpha_{j}=\sigma^j_{*}(\alpha_0)\circ \sigma^{j-1}_*(\sigma_{\mathcal{F}})\circ\ldots \circ\sigma_*(\sigma_{\mathcal{F}})\circ \sigma_{\mathcal{F}}$$
for all $j\geqslant 0$.
		This shows our adjunction:
		\[
		\Hom_{\sdx}\left(\mathcal{F},i^{!}(\mathcal{G})\right)\cong
		\Hom_{\sx}\left(i_*(\mathcal{F}),\mathcal{G}\right),
		\]
Analogously, we  define the functor $i^*$ by the formula:
\[
i^{*}(\mc{F}):=\bigoplus_{j=0}^{\infty} \sigma^{j*}(\mathcal{F})
\]
with the map
	\[\sigma^{i^{*}(\mc{F})}:\sigma^*(i^{*}(\mc{F}))=\bigoplus_{j=1}^{\infty} \sigma^{j*}(\mathcal{F})\ra
	i^{*}(\mc{F})=\bigoplus_{j=0}^{\infty} \sigma^{j*}(\mathcal{F})\]
	being the obvious embedding  of the subsum (this time we use that the functor $\sigma^*$ commutes with sums). The proof of the adjunction is analogous to the proof for $i^!$.

Item (3) is a formal consequence of item (2).

Item  (4) follows from the exactness of $i_*$ (hence it commutes with cohomology) and the fact that it does not kill objects.

Item $(5)$ (the exactness of $i^*$) follows from the construction of $i^*$ above and item (4).

For item (6), preserving injectives by $i^*$ and $i^!$ follows from the fact that they have exact left adjoints.
Finally, for an arbitrary $\mathcal{G}\in\mathrm{Sh}(\mathbf{C}_{\sigma}(X))$ we take an embedding
$\alpha: i_*(\mathcal{G})\ra\mcf$ into an injective sheaf. Then, the desired embedding is the following composition:
\begin{equation*}
\xymatrix{\mathcal{G} \ar[r]^{}  & i^! i_*(\mathcal{G}) \ar[r]^{i^!(\alpha)}  & i^!(\mcf),}
\end{equation*}
where the first map is the unit of adjunction. Indeed, the monomorphicity of the unit map follows from the construction of the functor $i^*$ above and the functor $i^!$ preserves monomorphisms, because it has a left adjoint.
\end{proof}

\section{Difference cohomology}\label{secleftdc}
In this section, we introduce the relevant cohomology theories, that is: the
sheaf cohomology (for difference sheaves) and the \v Cech cohomology (for difference presheaves).
Since our constructions do not depend on the specific properties of the underlying site $\mathbf{C}(X)$, we often remove $\mathbf{C}$ from our notation and, for example, abbreviate $\sdx$ just to $\dx$.

\subsection{Difference sheaf cohomology}\label{secldifcoh}
For any $\mcf\in\mathrm{Sh}(X)$, we have $\Gamma(\si_*(\mcf))=\Gamma(\mcf)$. Therefore, if
$(\mcf,\sigma_{\mathcal{F}})\in \dx$ then $\si_{\mcf}(X)$ acts on $\Gamma(\mcf)$.
We define now the {\it difference global section functor} as follows:
$$\Gamma^{\sigma}: \dx \ra \mathbf{Ab},\ \ \ \ \ \ \Gamma^{\sigma}(\mc{F}):=\mc{F}(X)^{\sigma_{\mcf}},$$
where $\mc{F}(X)^{\sigma_{\mcf}}$ denotes the invariants of the action of $\sigma_{\mc{F}}$ on $\Gamma(\mc{F})$. For a difference sheaf $(\mcf,\sigma_{\mathcal{F}})$, we often omit $\sigma_{\mathcal{F}}$ from the notation and simply write $\mcf\in \dx$.

Let $\Zz_X$ denote the difference constant sheaf on $X$ (see Example \ref{constex}). It is easy to see that there is a natural isomorphism:
\[
\Gamma^{\sigma}\cong \rm{Hom}_{\dx}(\Zz_X,-).
\]
Hence the functor $\Gamma^{\sigma}$ is left exact, and we can consider its right derived functors.
\begin{definition}
For $\mc{F}\in \dx$, we define the \emph{$n$-th difference sheaf cohomology}:
\[H^n_{\sigma}(X,\mc{F}):=R^n\Gamma^{\sigma}(X)={\rm Ext}^n_{\dx}(\Zz_X,\mc{F}).
\]
\end{definition}
We point out below a connection between the difference sheaf cohomology and the topological version of the discrete difference group cohomology, which was introduced in \cite{CK2}.
\begin{example}\label{sreptoet}
If $\mathbf{C}(X)$ and $\sigma$ are as in Section \ref{etaleexam}, then we have:
\[
H^n_{\sigma}(X,\mc{F})\cong H^n_{\sigma-\mathrm{top}}\left((\gal(\ka),\phi),(\Psi(\mathcal{F}),\Psi(\widetilde{\sigma_\mathcal{F}}))\right).
\]
\end{example}

In order to relate the difference sheaf cohomology to the ordinary sheaf cohomology, we need the result below.
\begin{lemma}\label{dscobv}
For $\mc{F}\in \dx$, we have the following.
\begin{enumerate}

\item The functor $\Gamma^{\sigma}$ can be factorized as the composition:
\begin{equation*}
\xymatrix{\dx \ar[r]^{\Gamma\ } & \modd_{\Zz[\sigma]} \ar[r]^{(-)^{\sigma}\ } & \modd_{\Zz}.}
\end{equation*}

\item The functor
$$\Gamma: \dx{\ra}\modd_{\Zz[\sigma]}$$
 preserves injectives.
\end{enumerate}
\end{lemma}
\begin{proof}
Item (1) is obvious. For item (2), it is easy to see that the construction from the moreover part of Example \ref{constex}
provides a left adjoint functor to the functor $\Gamma: \dx{\ra}\modd_{\Zz[\sigma]}$. Since the sheafification functor is exact, the constant difference sheaf functor described above is exact as well using Proposition \ref{improve}(5). Since any functor having exact left adjoint preserves injectives, item (2) follows. \end{proof}
We provide now a description of the difference cohomology, which will allow us to relate it to the usual sheaf cohomology and to the difference \v Cech cohomology (to be introduced in the next subsection). Let $\mcf\in\dx$  and  $\mcf\to I^{*}$ be  its injective resolution in $\dx$.
Then  $\sigma_{I^{*}}$ acts on $\Gamma(I^{*})$ and we define the bicomplex $D^{**}(\mathcal{F})$  with only two
non-trivial rows, both being $\Gamma(I^{*})$, with the vertical (upward) differential
$\pm(\id-\sigma_{I^{*}})$, and with the horizontal differential coming from the injective resolution $I^*$.
\begin{prop}\label{dcpxl}
 There is a natural in $\mathcal{F}$ isomorphism:
\[ H^*_{\sigma}(X,\mathcal{F})\cong H^*({\rm Tot}(D(\mathcal{F}))).\]	
\end{prop}
\begin{proof}
Since we have
$$H^0_{\sigma}(X,\mathcal{F})\cong H^0({\rm Tot}(D(\mathcal{F}))),$$
it is enough to show that  $H^*({\rm Tot}(D(\mathcal{F})))$ form a universal $\delta$-functor. By a functorial choice of the injective resolution $I^*$,  one  obtains the following functor:
\[\dx\ni \mcf\mapsto D^{**}(\mcf)\in \mathbf{Bicplx}\]
which is exact. Hence, a short exact sequence in $\dx$ gives rise to a long exact sequence of cohomology groups. It remains to show that
$H^s({\rm Tot}(D(\mathcal{F})))=0$ for $s>0$ and an injective $\mcf$. To this end, we look at the first spectral sequence (see e.g. \cite[Example 3, Appendix B]{milne1etale}) of the bicomplex $D^{**}(\mcf)$ for an injective $\mcf$. By Proposition \ref{improve}(6),
$i_*(\mcf)$ is injective, therefore, after taking the horizontal cohomology we are left with the complex:
\begin{equation*}
\xymatrix{0 \ar[r]^{} & \Gamma(\mcf) \ar[r]^{\id-\sigma_{\mcf}} \ar[r]^{} & \Gamma(\mcf) \ar[r]^{} & 0.}
\end{equation*}
Since the cohomology of this complex computes the groups
${\rm Ext}^*_{\Zz[\sigma]}(\Zz,\Gamma(\mcf))$, the result follows from Lemma \ref{dscobv}(2).
\end{proof}
In the following theorem (a sheaf analogue of \cite[Theorem 3.5]{CK2}),
$H^n(X,\mathcal{F})^{\sigma}$ and $H^n(X,\mathcal{F})_{\sigma}$ stands
for,  respectively, invariants and coinvariants, of the action
induced on cohomology by the action of $\si_{\mcf}$ on $\Gamma(\mcf)$.
\begin{theorem}\label{mainss}
For any $n\geqslant  0$, there is a short exact sequence:
\[0\ra H^{n-1}(X,\mathcal{F})_{\sigma}\ra H^n_{\sigma}(X,\mathcal{F})\ra H^n(X,\mathcal{F})^{\sigma}\ra 0,\]
where $H^{-1}(X,\mathcal{F}):=0$.
\end{theorem}
\begin{proof}
We immediately obtain this short exact sequences from the second page of the first spectral sequence associated to the bicomplex $D^{**}(\mcf)$ defined above. \end{proof}
As an example, we compute difference cohomology with constant coefficients.
\begin{example}\label{cocoef}
Let $A$ be a commutative group and $f:A\to A$ be a group endomorphism. By Example \ref{constex}, $(A,f_A)$ is a difference sheaf. By Theorem \ref{mainss}, we obtain the following short exact sequence:
\[0\ra H^{n-1}(X,A)_{\sigma}\ra H^n_{\sigma}(X,A)\ra H^n(X,A)^{\sigma}\ra 0.\]
In the case when $f$ is the identity map, we obtain above the (co-)invariants of the action of $\sigma:X\to X$ on cohomology. We would like to single out one more special case. If we assume that:
\begin{itemize}
\item $X$ is defined over $\Ff_p$ (finite field with $p$ elements);

\item $\sigma=\fr_X$ is the geometric Frobenius morphism on $X$;

\item $\mathbf{C}(X)=\mathbf{E}(X)$ (\'{e}tale topology);

\item $A$ is a vector space over a field $\ka$;

\item $f(x)=\alpha x$ is a scalar multiplication for a fixed $\alpha\in \ka$;
\end{itemize}
then we are in the situation of Lefschetz trace formula, see \cite[Section V.2]{milne1etale}.
\end{example}

\subsection{Difference \v Cech cohomology}\label{leftcechsec}
Let $(\mcf,\si_{\mcf})$ be a difference  presheaf and
$\mc{U}=(U_i\to X)_{i\in I}$ be a covering of $X$. Then $\si_{\mcf}$ gives rise to a map between the \v Cech complexes
\[\check{C}^*(\mc{U},\mcf)\ra \check{C}^*({}^{\sigma}\mc{U},\mcf).
\]
The composition of this map with the obvious restriction map is denoted as follows:
\[\check{\sigma}_{\mcf}:\check{C}^*(\mc{U},\mcf)\ra \check{C}^*(\mc{U}\cap{}^{\sigma}\mc{U},\mcf),
\]
where $\mc{U}\cap{}^{\sigma}\mc{U}$ denotes the covering of $X$ consisting of $V\times_X{}^{\sigma}U$ for $V,U\in \mathcal{U}$. We introduce the following notation
\[\mathrm{res}:\check{C}^*(\mc{U},\mcf)\ra \check{C}^*(\mc{U}\cap{}^{\sigma}\mc{U},\mcf)
\]
for the restriction map from $\mc{U}$ to $\mc{U}\cap{}^{\sigma}\mc{U}$.

Analogously to the bicomplex $D^{**}(\mathcal{F})$ from the end of Section \ref{secldifcoh}, we define the bicomplex
$\check{C}^{**}_{\sigma}(\mathcal{U},\mc{F})$ by putting:
$$\check{C}^{*0}_{\sigma}(\mathcal{U},\mc{F}):=\check{C}^{*}(\mathcal{U},\mc{F}),$$
$$\check{C}^{*1}_{\sigma}(\mathcal{U},\mc{F}):=\check{C}^{*}(\mathcal{U}\cap{}^{\sigma}\mc{U},\mc{F}),$$
and trivial elsewhere. The horizontal differential comes from the standard \v{C}ech complex and the vertical differential is the map $\pm({\rm res}-\check{\sigma}_{\mcf})$. We call this bicomplex the \emph{difference \v Cech bicomplex} and we call its cohomology the \emph{difference \v Cech cohomology}, which will be denoted by
$\check{H}^*_{\sigma}(\mc{U},\mcf)$.

Similarly as in the case of the bicomplex $D^{**}(\mathcal{F})$, assigning the difference \v Cech bicomplex is an exact functor, hence the following functor
\[\mathrm{PSh}_{\sigma}(X)\ni \mcf\mapsto \check{H}^0_{\sigma}(\mc{U},\mcf)\in \mathbf{Ab}.\]
is left exact. We show below that, similarly to the classical context, the difference \v Cech cohomology groups coincide with the right derived functors of this functor.
\begin{prop}\label{cechder}
There is a natural isomorphism
\[
\check{H}^s_{\sigma}(\mc{U},\mcf)\cong R^s\left(\check{H}^{0}_{\sigma}(\mathcal{U},-)\right)(\mc{F}).
\]
\end{prop}
\begin{proof}
It is  straightforward that the cohomology groups of difference \v {C}ech bicomplex form a $\delta$-functor, so it is enough to show that the higher cohomology group functors  vanish on injective cogenerators. By the moreover clause in Proposition \ref{improve}(6), we can focus on the difference sheaves of the form $i^!(\mcf)$ for an injective $\mcf\in \rm{PSh}(X)$. To simplify the notation, we abbreviate below $U_{i_1}\times_X U_{i_2}$ to $U_{i_1 i_2}$, ${}^{\sigma^n}(U_i)$
 to $U_i^n$, and $\mathbf{i}=(i_1,\ldots ,i_k)$.

To some extent, we follow the lines of the proof of  \cite[Theorem I.2.2.3]{tamme}, in particular we will use the presheaves of the form $\Zz_U$. By the explicit description of the difference sheaf $i^!(\mcf)$ given in the proof of Proposition \ref{improve}, we have the following:
 \begin{IEEEeqnarray*}{rCl}
\check{C}^{k0}_{\sigma}\left(\mathcal{U},i^!(\mc{F})\right) & = & \prod_{I^k} i^!(\mcf)(U_{\mathbf{i}}) \\
 &\cong &  \prod_{I^k\times\Nn}\mcf(U_{\mathbf{i}}^n)\\
 &\cong & \Hom_{\mathrm{PSh}(X)}(\bigoplus_{I^k\times\Nn}\Zz_{U_{\mathbf{i}}^n},\mcf).
\end{IEEEeqnarray*}
We similarly obtain:
 \[\check{C}^{k1}_{\sigma}\left(\mathcal{U},i^!(\mc{F})\right)\cong \Hom_{\mathrm{PSh}(X)}(\bigoplus_{I^k\times I^k\times\Nn}
 \Zz_{U_{\mathbf{i}_1}^n\times_X U_{\mathbf{i}_2}^{n+1}},\mcf).
 \]
The vertical differential is induced by the restriction maps coming from the projections in
the pull-backs $U_{\mathbf{i}_1}^n\times_X  U_{\mathbf{i}_2}^{n+1}$.

Since $\mcf$ is injective, it suffices to show that the bicomplex with the following two non-trivial rows:
$$\left(\bigoplus_{I^k\times\Nn}\Zz_{U_{\mathbf{i}}^n},\bigoplus_{I^k\times I^k\times\Nn}
 \Zz_{U_{\mathbf{i}_1}^n\times_X U_{\mathbf{i}_2}^{n+1}}\right)_{k\in \Nn}$$
has no higher homology. To this end, it suffices to show this for its evaluation on any $V\in \mathbf{C}(X)$. Similarly as in the proof of  \cite[Theorem I.2.2.3]{tamme}, after this evaluation the non-trivial rows of our bicomplex have the following form:
$$\left(\bigoplus_{n=0}^{\infty}\bigoplus_{\Sigma_n^k}\Zz,\bigoplus_{n=0}^{\infty}\bigoplus_{\Sigma_n^k\times \Sigma_{n+1}^k}\Zz\right)_{k\in \Nn},$$
where
$$\Sigma_n:=\coprod_I \Hom_{\mathbf{C}(X)}(V,U_i^n),$$
and  $\Sigma_n^k$ is the $k$-th Cartesian power.

We look at the first spectral sequence associated to this last bicomplex above. We observe
 that the horizontal differential preserves the index $n$, hence this bicomplex splits into a direct sum of usual \v Cech complexes associated to the coverings $\mathcal{U}^n$ in the $0$-th row and to the coverings $\mathcal{U}^n\times_X\mathcal{U}^{n+1}$ in the first row. Thus we get
 \[H^{00}_{\phi, \mathrm{hor}}=\bigoplus_{n=0}^{\infty} \Zz,\ \ \ \ \ H^{01}_{\phi, \mathrm{hor}}=\bigoplus_{n=0}^{\infty} \Zz\]
 and $H^{0*}_{\phi ,\mathrm{hor}}$ is obtained from $C^{0*}_{\phi, \mathrm{hor}}$ by identifying all the copies of
 $\Zz$ (this is how the \v Cech differential acts). Thus the vertical differential map
 \[
 \bigoplus_{n=0}^{\infty} \Zz\ra \bigoplus_{n=0}^{\infty} \Zz
 \]
 sends $1_n$, the generator of the $n$-th copy  of $\Zz$, to $1_n-1_{n+1}$. Hence, this last map is monomorphic and there is no higher homology.
\end{proof}
Similarly to the classical context again, the difference \v Cech cohomology and the difference sheaf cohomology are related via a spectral sequence. For a difference presheaf $\mcf$, we define:
$$\mathcal{H}^q(\mathcal{F})(U):=H^q(X,\mcf|_{U}).$$
Then $\mathcal{H}^q(\mathcal{F})$ is a difference presheaf and we have the following.
 \begin{prop}\label{quickss}
For any  difference sheaf $\mathcal{F}$, there is a spectral sequence:
\[E^{pq}_2=\check{H}_{\sigma}^p(\mathcal{U},\mc{H}^q(\mcf)) \Rightarrow  H^{p+q}_{\sigma}(X,\mcf).\]
\end{prop}
\begin{proof}
We have a natural isomorphism of functors:
$$\check{H}^{0}_{\sigma}(\mathcal{U},-)\circ o\cong \Gamma^{\sigma},$$
where $o$ is the difference forgetful functor from Remark \ref{leftoa}. By  Remark \ref{leftoa}(3), the functor $o$ has a left adjoint (the difference sheafification), which is exact by Proposition \ref{improve}(4) (both the sheaf and the presheaf version).
Therefore, the functor $o$ preserves injectives. Thus we obtain our spectral sequence as the Grothendieck
spectral sequence.

To see that this spectral sequence converges to the correct object, it is enough to notice that $\mathcal{H}^q\cong R^q(o)$, which follows from its classical counterpart and the presheaf variant of Proposition \ref{improve}(6).
\end{proof}
The next result is a \v Cech cohomology version of Theorem \ref{mainss}.
\begin{theorem}\label{cechss}
For any $n\geqslant  0$,
there is a short exact sequence
	\[0\ra \check{H}^{n-1}(\mathcal{U},\mathcal{F})_{\sigma}\ra \check{H}^n_{\sigma}(\mathcal{U},\mathcal{F})\ra \check{H}^n(\mathcal{U},\mathcal{F})^{\sigma}\ra 0,\]
 where $\check{H}^{-1}(\mathcal{U},\mathcal{F}):=0$.
\end{theorem}
\begin{proof}
Similarly as in the proof of Theorem \ref{mainss}, the existence of this short exact sequence follows immediately
from the first spectral sequence for the difference \v Cech bicomplex.
\end{proof}
At last we would like to compare difference \v Cech and difference sheaf cohomology. Let $\beta$ be the classical map from \v Cech to sheaf cohomology (see e.g. \cite[(I.3.4.5)]{tamme}).
\begin{theorem}\label{cechtoder}
We have  a natural map:
$$\alpha:\check{H}_{\sigma}^p(\mathcal{U},\mathcal{F})\ra H_{\sigma}^p(X,\mathcal{F}),$$
which  fits into the following commutative diagram:
\begin{equation*}
\xymatrix{0  \ar[r]^{}  & \check{H}^{n-1}(\mathcal{U},\mathcal{F})_{\sigma} \ar[r]^{} \ar[d]^{\beta_{\sigma}} & \check{H}^n_{\sigma}(\mathcal{U},\mathcal{F}) \ar[r]^{} \ar[d]^{\alpha} & \check{H}^n(\mathcal{U},\mathcal{F})^{\sigma} \ar[r]^{} \ar[d]^{\beta^{\sigma}} & 0 \\
0  \ar[r]^{}  & H^{n-1}(X,\mathcal{F})_{\sigma} \ar[r]^{} & H^n_{\sigma}(X,\mathcal{F}) \ar[r]^{} &  H^n(X,\mathcal{F})^{\sigma} \ar[r]^{} & 0 ,}
\end{equation*}
whose rows are given by Theorem \ref{mainss} and Theorem \ref{cechss}.

Therefore, if for a given $\mcf$, $\beta$ is an isomorphism in degrees $n-1$ and $n$, then $\alpha$ is an isomorphism for this $\mcf$ in degree $n$.
\end{theorem}
\begin{proof}
We observe that $\beta$  induces the map between the first pages of the first spectral sequences of the corresponding  bicomplexes (see Proposition \ref{dcpxl})
$$ E_1(\check{C}^{**}(\mathcal{U},\mcf))\mapsto E_1(D^{**}(\mcf)),$$
 hence our
$\alpha$ together with the required compatibility comes as the map induced on the limits of these spectral sequences.
\end{proof}
We define now the limit  difference \v Cech cohomology:
\[
\check{H}^n_{\sigma}(X,\mcf):=
\coli_{\mathcal{U}}\check{H}^n_{\sigma}(\mc{U},\mcf),
\]
where the direct limit runs over the family of coverings of $X$.
In order to establish standard properties of the limit difference \v Cech cohomology we need, analogously to the classical context, the lemma allowing one to replace the system of coverings
with the cofiltered one.
\begin{lemma}\label{indep}
Let $f,g:\mathcal{U}\to\mathcal{V}$ be two refinement maps. Then they induce the same map
on the \v Cech cohomology.
\end{lemma}
\begin{proof}
By the standard degree-shift argument, it suffices to show our assertion for
$\check{H}^0_{\sigma}$. However, since $\check{H}^0_{\sigma}(\mathcal{V},\mathcal{F})\subset
\check{H}^0(\mathcal{V},\mathcal{F})$ for any difference presheaf $\mathcal{F}$, our claim follows from the analogous fact for the usual \v Cech cohomology. \end{proof}
Once we have this lemma, we immediately deduce the statements below concerning limit difference \v Cech cohomology from their non-limit counterparts.
\begin{prop}\label{cechder2}
There is a natural isomorphism
\[
\check{H}^s_{\sigma}(X,\mcf)\cong R^s\left(\check{H}^{0}_{\sigma}(X,-)\right)(\mc{F}).
\]
\end{prop}
 \begin{prop}\label{quickss2}
For any left difference sheaf $\mathcal{F}$, there is a spectral sequence:
	\[E^{pq}_2=\check{H}_{\sigma}^p(X,\mc{H}^q(\mcf)) \Rightarrow   H^{p+q}_{\sigma}(X,\mcf).\]
\end{prop}
\begin{theorem}\label{cechss2}
For any $n\geqslant  0$,
there is a short exact sequence:
	\[0\ra \check{H}^{n-1}(X,\mathcal{F})_{\sigma}\ra \check{H}^n_{\sigma}(X,\mathcal{F})\ra \check{H}^n(X,\mathcal{F})^{\sigma}\ra 0,\]
 where $\check{H}^{-1}(X,\mathcal{F}):=0$.
\end{theorem}

\begin{theorem}\label{cechtoder2}
We have  a natural map:
$$\alpha:\check{H}_{\sigma}^p(X,\mathcal{F})\ra H_{\sigma}^p(X,\mathcal{F}),$$
which  fits into the  commutative diagram with exact rows:
\begin{equation*}
\xymatrix{0  \ar[r]^{}  & \check{H}^{n-1}(X,\mathcal{F})_{\sigma} \ar[r]^{} \ar[d]^{\beta_{\sigma}} & \check{H}^n_{\sigma}(X,\mathcal{F}) \ar[r]^{} \ar[d]^{\alpha} & \check{H}^n(X,\mathcal{F})^{\sigma} \ar[r]^{} \ar[d]^{\beta^{\sigma}} & 0 \\
0  \ar[r]^{}  & H^{n-1}(X,\mathcal{F})_{\sigma} \ar[r]^{} & H^n_{\sigma}(X,\mathcal{F}) \ar[r]^{} &  H^n(X,\mathcal{F})^{\sigma} \ar[r]^{} & 0 ,}
\end{equation*}
where $\beta$ is the usual  map from \v Cech to sheaf cohomology \cite[(I.3.4.5)]{tamme}. Therefore, if for a given $\mcf$, $\beta$ is an isomorphism in degrees $n-1$ and $n$, then $\alpha$ is an isomorphism for this $\mcf$ in degree $n$.

In particular, $\alpha$ is always an isomorphism for $n=0,1$, or for any $n$ when
 $\mathbf{C}(X)=\mathbf{Z}(X)$, $X$ is a separated scheme, and $\mathcal{F}$ is a  difference quasi-coherent sheaf (see \cite[Proposition III.2.14]{milne1etale}).
\end{theorem}

\section{Difference torsors}\label{secdiftor}
In this section, we develop a general theory of difference torsors for left difference sheaves. We still work under Assumption \ref{mainass}. Since the choice of a site does not really matter for the very general Section \ref{ldstsec} below, we still skip ``$\mathbf{C}$'' from the notation there. However, the choice of site becomes crucial starting from Section \ref{contor} and we will change our notation there.

\subsection{Difference sheaf torsors}\label{ldstsec}
Let $\mc{G}$ be a  sheaf of (not necessarily abelian) groups on $\mathbf{C}(X)$ (we will just say ``sheaf on $X$'', it may be even a sheaf of sets). We recall here some definitions from \cite[Section III.4]{milne1etale}.

We say that  a sheaf of sets  $\mc{P}$ on $X$ is a \emph{sheaf of $\mc{G}$-sets}
(or a \emph{$\mc{G}$-sheaf}), if there is a morphism of sheaves
$$\mu:\mc{G}\times \mc{P}\ra\mc{P}$$
such that for any $U\in{\mathcal{C}}$, $\mu$ induces an action of the group $\mc{G}(U)$ on the set $\mc{P}(U)$, and these actions are compatible with the restriction maps. We say that a sheaf of $\mc{G}$-sets $\mc{P}$ is a \emph{trivial $\mc{G}$-torsor}, if $\mc{P}\cong \mc{G}$ as $\mc{G}$-sheaves.
We say that $\mc{P}$ is a \emph{sheaf $\mc{G}$-torsor}, if there is a covering $\{U_i\to X\}_{i\in I}$ such that for any $i\in I$, $\mc{P}|_{U_i}$ is a trivial $\mc{G}|_{U_i}$-torsor. The sheaf $\mc{G}$-torsors form the category $\mathrm{TSh}(\mc{G}/X)$ with the morphisms being the $\mc{G}$-invariant sheaf morphisms. We denote the set of isomorphism classes of sheaf $\mc{G}$-torsors by $\mathrm{PHSh}(\mc{G}/X)$.

If we have a scheme morphism $f:X\to X'$ (and there is an underlying site $\mathbf{C}'(X')$ such that $f$ induces a morphism of sites), then we get the pull-back functor:
\[
f^*:\mathrm{TSh}({\mc{G}}/X')\ra \mathrm{TSh}({f^*(\mc{G})}/X).
\]
A morphism of group sheaves $\alpha:\mc{G}\to\mc{H}$ on $X$ gives rise to the extension functor:
\[
\alpha_*:\mathrm{TSh}(\mc{G}/X)\ra \mathrm{TSh}({\mc{H}}/X),\ \ \ \ \ \alpha_*(\mathcal{P})(U)=\left(\mc{H}(U)\times \mathcal{P}(U)\right)/\mc{G}(U).
\]
We recall below the standard correspondence between torsors and \v Cech cohomology.
\begin{lemma}\label{clphs}
There is a bijection
$$h:\mathrm{PHSh}(\mc{G}/X)\ra \check{H}^1(X,\mathcal{G})$$
such that the following hold (we identify below a torsor with the corresponding isomorphism class).
\begin{enumerate}
\item $h(f^*(\mathcal{P}))=\check{H}^1(f,\mathcal{G})(h(\mathcal{P}))$;

\item $h(\alpha_*(\mathcal{P}))=\check{H}^1(X,\alpha)(h(\mathcal{P}))$.
\end{enumerate}
\end{lemma}
\begin{proof}
The bijection above appears in the statement of \cite[Proposition III.4.6]{milne1etale}. From the explicit construction in the proof of \cite[Proposition III.4.6]{milne1etale}, one gets items $(1)$ and $(2)$.
\end{proof}
We recall now that we are in the situation of Assumption \ref{mainass}, and we let $\sigma$ into the play.
Let $\mc{P}$ be a sheaf of $\mc{G}$-sets on $X$. Then $\sigma^*(\mc{P})$ is a sheaf of $\sigma^*(\mc{G})$-sets.
\begin{definition}\label{phssdef}
We assume that $(\mc{G},\sigma_{\mc{G}})$ is a left difference sheaf of groups (not necessarily abelian groups!) and $\mc{P}$ is a sheaf $\mc{G}$-torsor.
\begin{enumerate}
\item We introduce the following functor:
\[\sigma_T:\mathrm{TSh}(\mc{G}/X)\ra\mathrm{TSh}(\mc{G}/X),\ \ \ \ \ \   \sigma_T(\mc{P}):=\left(\widetilde{\sigma}_{\mc{G}}\right)_*(\sigma^*(\mc{P})).\]

\item A \emph{left difference sheaf $\mc{G}$-torsor} is a pair $(\mc{P},\sigma_{\mc{P}})$, where $\mc{P}\in \mathrm{TSh}(\mc{G}/X)$
and
\[
\sigma_{\mc{P}}:\mc{P} \ra \sigma_T(\mc{P})
\]
is an isomorphism of sheaf $\mc{G}$-torsors.

\item The left difference sheaf $\mc{G}$-torsors form a category and we denote the set of isomorphism classes of left difference sheaf $\mc{G}$-torsors by $\mathrm{PHSh}_{\sigma}(\mc{G}/X)$.

\end{enumerate}
\end{definition}
For a covering $\mathcal{U}$ of $X$ and $\mathcal{F}\in \dx$, it is easy to describe explicitly the total complex of the difference \v Cech bicomplex $\check{C}^{**}_{\sigma}(\mathcal{U},\mathcal{F})$, which defines the difference \v Cech cohomology $\check{H}^{*}_{\sigma}(\mathcal{U},\mathcal{F})$. This description (given in the remark below) will be also used to \emph{define} the first difference cohomology pointed set for a difference sheaf of not necessarily commutative groups.
\begin{remark}\label{totdiff1}
To simplify the notation, we write here $\check{C}^{n}$ instead of $\check{C}^{n}(\mathcal{U},\mathcal{F})$.
\begin{enumerate}
\item We have the following description of the total complex of the difference \v Cech bicomplex:
$$\mathrm{Tot}\left(\check{C}^{**}_{\sigma}\right)^n=\check{C}^{n-1} \oplus \check{C}^{n},$$
$$\check{\partial}^{n}_{\sigma-\mathrm{tot}}(c^{n-1},c^n)=\left(\check{\partial}^{n-1}(c^{n-1})+(-1)^n(\mathrm{res}-\check{\sigma})(c^n),\check{\partial}^{n}(c^n)\right),$$
where $\check{C}^{-1}:=0$ and $\check{\partial}$ is the differential from the standard \v{C}ech complex.

\item It is clear from the previous item that for $(c^{n-1},c^n)\in \mathrm{Tot}(\check{C}^{**}_{\sigma})^n$, we have that $(c^{n-1},c^n)\in \mathrm{Tot}(\check{Z}^{**}_{\sigma})^n$ if and only if the following two conditions hold:
    \begin{itemize}
      \item $c^n$ is a standard \v Cech cocycle;

      \item $(\mathrm{res}-\check{\sigma})(c^n)$ is a standard \v Cech coboundary which is witnessed by $(-1)^{n-1}c^{n-1}$.
    \end{itemize}
\end{enumerate}
\end{remark}

\begin{definition}\label{defcnon}
Let $\mathcal{G}$ be a left difference sheaf of not necessarily commutative groups. We define the first pointed set difference \v Cech cohomology $\check{H}^{1}_{\sigma}(\mathcal{U},\mathcal{G})$ as the set of equivalence classes of the relation:
$$(c^0,c^1)\sim (e^0,e^1) \ \ \ \Leftrightarrow\ \ \ \exists f^0\in \check{C}^{0}(\mathcal{U},\mathcal{G})\
\left(\mathrm{res}/\check{\sigma}\right)(f^0)=c^0/e^0,\ \check{\partial}^0(f^0)=c^1/e^1.$$
on the set $\check{C}^{0}(\mathcal{U},\mathcal{G})\times \check{C}^{1}(\mathcal{U},\mathcal{G})$. We also obtain the pointed set  $\check{H}^{1}_{\sigma}(X,\mathcal{G})$ as the usual colimit over all the coverings $\mathcal{U}$.
\end{definition}
By Remark \ref{totdiff1}, if $\mathcal{G}\in \dx$ then Definition \ref{defcnon} gives the same notion as the definition of the difference \v Cech cohomology from Section \ref{leftcechsec}.

We obtain below a difference version of the classical correspondence between torsors and the first cohomology group.
\begin{theorem}\label{lphsthm}
Let $\mathcal{G}$ be a left difference sheaf of not necessarily commutative groups. Then, there is an isomorphism of pointed sets, which is an isomorphism of abelian grups when $\mc{G}$ is commutative:
\[
\mathrm{PHSh}_{\sigma}(\mc{G}/X)\cong \check{H}^1_{\sigma}(X,\mc{G}).
\]
\end{theorem}
\begin{proof}
Let $(\mc{P},\sigma_{\mc{P}})$ be a left difference sheaf $\mc{G}$-torsor and $c(\mc{P})\in \check{Z}^1(X,\mc{G})$ be a cocycle corresponding to the cohomology class $h(\mc{P})$ from Lemma \ref{clphs}. The fact that the $\mc{G}$-torsors $\mc{P}$ and $\sigma_T(\mc{P})$ are isomorphic means that the \v Cech cocycle $(\id-\check{\sigma})(c(\mc{P}))$ is a \v Cech coboundary, and the choice of the isomorphism $\sigma_{\mc{P}}:\mc{P}\to \sigma_T(\mc{P})$ corresponds
to  the choice of $c(\sigma_{\mc{P}})\in\check{C}^0(X,\mc{G}) $ such that:
$$\check{\partial}^0\left(c(\sigma_{\mc{P}})\right)=(\mathrm{res}-\check{\sigma})(c(\mc{P})).$$
By Remark \ref{totdiff1}(2) (or Definition \ref{defcnon}), the pair $(c(\mc{P}),c(\sigma_{\mc{P}}))$ is a $1$-cocycle in the total complex of $\check{C}^{**}_{\sigma}(X,\mc{G})$, and it is a matter of routine verification to check that this construction gives the desired
isomorphism.
\end{proof}
The short exact sequence from Theorem \ref{cechss} (relating the difference cohomology with the standard cohomology) may be interpreted geometrically here. We need one more definition.

\begin{definition}\label{asdef}
Let $G$ be a group with an endomorphism $s$.
\begin{enumerate}
\item If $G$ is commutative (the group operation is written additively), then we define:
$$\mathrm{AS}(G,s):=\coker(s-\id)=G_s,$$
where $G_s$ denotes the coinvariants of the action on $G$ by $s$. In this case, we also call the map $s-\id$, the \emph{Artin-Schreier map}.

\item If $G$ is not commutative, then
$\mathrm{AS}(G,s)$ is defined as the set of orbits of the following action (of $G$ on $G$):
\[g\cdot x:=s(g) x g^{-1},\]
which clearly coincides with the definition given in the commutative case.
\end{enumerate}
\end{definition}
\begin{example}\label{asexam}
\begin{enumerate}
\item Let $(\mathcal{G},\sigma_{\mathcal{G}})$ be a left difference sheaf of (not necessarily commutative) groups. Then $\Gamma(\sigma_{\mathcal{G}})$ is an endomorphism of $\Gamma(\mathcal{G})$ and we denote:
$$\mathrm{AS}(\mathcal{G},\sigma_{\mathcal{G}}):=\mathrm{AS}\left(\Gamma(\mathcal{G}),\Gamma(\sigma_{\mathcal{G}})\right).$$

\item Let $G$ be a group scheme over $X$ defined over constants of $\sigma$ (see Section \ref{igs}). Then $\sigma$ induces a group homomorphism $G(\sigma):G(X)\to G(X)$, and we define:
    $$\mathrm{AS}(G,\sigma):=\mathrm{AS}(G(X),G(\sigma)).$$
If $G=\ga$, $X=\spec(\ka)$, $\mathrm{char}(\ka)>0$ and $\sigma$ comes from $\fr_{\ka}$, then we obtain the original ``Artin-Schreier situation''.
\end{enumerate}
\end{example}
In the following result, we use the notation from Example \ref{asexam}(1).
\begin{prop}\label{asphs}
There exists a short exact sequence (of pointed sets, in the case when $\mathcal{G}$ is a sheaf of non-commutative groups):
\[
1\ra \mathrm{AS}\left(\mathcal{G},\sigma_{\mathcal{G}}\right)\ra \mathrm{PHSh}_{\sigma}(\mc{G}/X)\ra \mathrm{PHSh}(\mc{G}/X)^{\sigma_T}\ra 1.
\]
\end{prop}
\begin{proof}
Let
$$\alpha:\mathrm{PHSh}_{\sigma}(\mc{G}/X)\ra \mathrm{PHSh}(\mc{G}/X)^{\sigma_T}$$
be the obvious epimorphism of forgetting the left difference structure. We only need to describe $\ker(\alpha)$, that is to prove that the difference structures on the trivial $\mc{G}$-torsor are classified by $\mathrm{AS}\left(\mathcal{G},\sigma_{\mathcal{G}}\right)$. Hence, we need to classify the $\mathcal{G}$-torsor endomorphisms $s:\mathcal{G}\to \mathcal{G}$ up to isomorphisms of them. Such endomorphisms $s$ are in a bijection with $\Gamma(\mathcal{G})$ via:
$$s\mapsto s_X\left(1_{\Gamma(\mathcal{G})}\right).$$
Two elements $g_1,g_2\in \Gamma(\mathcal{G})$ correspond to isomorphic $\mathcal{G}$-torsor endomorphisms
 if and only if there exists $g\in\Gamma(\mc{G})$ such that
 \[
 g_1 \cdot \Gamma(\sigma_{\mathcal{G}})(g)=g \cdot g_2,
 \]
which gives our assertion.
\end{proof}
\begin{remark}
In the case of a non-commutative $\mc{G}$, there is some extra information, which was not mentioned in the statement of Proposition \ref{asphs}. Namely, the fiber of the epimorphism $\alpha$ (see the proof of Proposition \ref{asphs}) over the isomorphism class of a torsor $\mathcal{P}$
is isomorphic as an $\mathrm{Aut}(\mathcal{P})$-set to $\mathrm{AS}\left(\mathrm{Aut}(\mathcal{P}), \bar{\sigma}_{\mathcal{P}}\right)$, where the endomorphism $\bar{\sigma}_{\mathcal{P}}$ is given by the following formula:
 \[
 f\mapsto   (\sigma_{\mathcal{P}})^{-1}\circ \sigma_T(f)\circ \sigma_{\mathcal{P}}.
 \]
 \end{remark}

\subsection{Difference torsors defined over constants}\label{contor}
We focus now on a certain special case which covers many of the examples occurring in practice.
We assume that $G$ is a flat group scheme over $X$. Then there is a classical notion of a scheme $G$-torsor over $X$ (locally trivial in the flat topology $\mathbf{F}(X)$), see \cite[Chapter III.3]{milne1etale}. We denote the category of scheme $G$-torsors over $X$ by $\mathrm{Tors}(G/X)$, and the pointed set (the commutative group, if $G$ is commutative) of isomorphism classes of scheme $G$-torsors over $X$ by $\mathrm{PHS}(G/X)$. Similarly as in the previous subsection, we have the following functor induced by $\sigma$:
$$\mathrm{Tors}(G/X)\ni P\mapsto {}^{\sigma}P\in \mathrm{Tors}({}^{\sigma}G/X),$$
and if $f:H\to G$ is a morphism group schemes over $X$, then we get the extension functor:
$$f_*:\mathrm{Tors}(H/X)\ra \mathrm{Tors}(G/X).$$
We recall from Section \ref{secintro} that for an $X$-scheme $Y$, $\mathcal{R}(Y)$ denotes the corresponding representable functor, which is a sheaf thanks to Assumption \ref{mainass}(3). Hence, we have one more functor:
$$\mathcal{R}:\mathrm{Tors}(G/X)\ra \mathrm{TSh}(\mathcal{R}(G)/\mathbf{F}(X)).$$
Let us assume now that $G$ is defined over constants (see Section \ref{igs}). Then, there is an isomorphism $\tau:{}^{\sigma}G\to G$ of group schemes over $X$ (see Lemma \ref{defoverc}). We recall that the representable sheaf $\mathcal{R}(G)$ has a natural structure of a left difference sheaf, which is given by the following commutative diagram:
\begin{equation*}
\xymatrix{\sigma^*(\mathcal{R}(G))  \ar[rr]^{\widetilde{\sigma}_{\mathcal{R}(G)}}  \ar[rd]_{\phi_G} & & \mathcal{R}(G)  \\
& \mathcal{R}({}^{\sigma}G). \ar[ru]_{\mathcal{R}(\tau)}  & & }
\end{equation*}
We assume now that $G$ satisfies one of the assumptions of \cite[Theorem III.4.3]{milne1etale}, which guarantee that any
sheaf of $\mathcal{R}(G)$-torsors on $\mathbf{F}(X)$ comes from a ``scheme $G$-torsor'' as in \cite[Prop. III.4.1]{milne1etale}, that is the map:
$$\mathcal{R}:\mathrm{PHS}(G/X)\ra \mathrm{PHSh}(\mathcal{R}(G)/\mathbf{F}(X)).$$
is a bijection. We aim to extend this bijective correspondence to the difference context (it will done in a much greater generality in Section \ref{secrdt}, however, this greater generality will require the notion of a \emph{right} difference sheaf).
\begin{definition}\label{defdgt}
A pair $(P,\sigma_P)$ is called a \emph{difference $G$-torsor} (it is a special case of Definition \ref{dtordef}), if the following holds:
\begin{itemize}
\item $P$ is a scheme $G$-torsor over $X$;

\item  $\sigma_P:P\to \sigma_{\tau}(P)$ is a morphism of $G$-torsors over $X$, where
$$\sigma_{\tau}(P):=\tau_*({}^{\sigma}P).$$
\end{itemize}
We denote the set of isomorphism classes of difference $G$-torsors by $\mathrm{PHS}_{\sigma}(G/X)$.
\end{definition}
The following example comes from \cite[Example 1.4]{BW}, however, it is slightly modified here to fit into our terminology. More comments on the set-up from \cite{BW}, and the comparison to the situation from this paper will be given in Remark \ref{remdifsch}.
\begin{example}\label{exbw1}
Let $(\ka,s)$ be a difference field and $X:=\spec(\ka),\sigma:=\spec(s)$. We assume that $\mathrm{char}(\ka)\neq 2$ and consider the group scheme $G:=\mu_2$ (second roots of unity) over $\ka$, which is clearly defined over constants. Following \cite[Example 1.4]{BW}, for $a,b\in \gm(\ka)$ such that $s(a)=ab^2$, we define the difference $\mu_2$-torsor $(\mu_{2,a},\varphi_b)$ in the following way. Firstly:
$$\mu_{2,a}:=\spec\left(\ka[x]/(x^2-a)\right)$$
is an algebraic $\mu_2$-torsor. Since $s(a)=ab^2$, we get:
$${}^{\sigma}\mu_{2,a}=\spec\left(\ka[x]/(x^2-ab^2)\right),$$
and we define:
$$\varphi_b:\mu_{2,a}\ra {}^{\sigma}\mu_{2,a},\ \ \ \ \ \ \ x\mapsto bx.$$
As proved in \cite[Example 5.3]{BW}, the above difference torsors exhaust all the possibilities (up to an isomorphism), which we will also show in Example \ref{exbw1calc} by a quick application of our general methods.
\end{example}
Thanks to the following result, we can apply the results of Section \ref{ldstsec}.
 \begin{lemma}\label{torbij}
Difference $G$-torsors are in a natural bijection with left difference sheaf $\mathcal{R}(G)$-torsors with respect to the flat site $\mathbf{F}(X)$.
 \end{lemma}
\begin{proof} Let $(P,\sigma_P)$ be a difference $G$-torsor. Applying the functor $\mathcal{R}$, we get a pair $(\mathcal{R}(P),\mathcal{R}(\sigma_P))$, where
$$\mathcal{R}(\sigma_P):\mathcal{R}(P)\ra \mathcal{R}(\tau_*({}^{\sigma}P))=\mathcal{R}(\tau)_*(\mathcal{R}({}^{\sigma}P))$$
and the equality follows, since the representability functor $\mathcal{R}$ commutes with the extension functors. It is enough to show that the pair $(\mathcal{R}(P),\mathcal{R}(\sigma_P))$ is a left difference sheaf $\mathcal{R}(G)$-torsor in the sense of Definition \ref{phssdef}(2). Hence, it is enough to notice that:
$$\mathcal{R}(\tau)_*(\mathcal{R}({}^{\sigma}P))= \left(\widetilde{\sigma}_{\mathcal{R}(G)}\right)_*(\sigma^*(\mc{P})),$$
which can be easily observed by chasing an appropriate diagram.
\end{proof}
Using Theorem \ref{cechtoder2}, Theorem \ref{lphsthm} and Lemma \ref{torbij}, we obtain the following.
\begin{theorem}\label{dtiso}
	There is the following isomorphism of pointed sets:
	\[
	\mathrm{PHS}_{\sigma}(G/X)\cong \check{H}^1_{\sigma}(\mathbf{F}(X),\mathcal{R}(G)).
	\]
If $G$ is commutative, then we also get the following isomorphism of abelian groups:
\[
	\mathrm{PHS}_{\sigma}(G/X)\cong H^1_{\sigma}(\mathbf{F}(X),\mathcal{R}(G)).
	\]
\end{theorem}
\begin{remark}\label{ftoetoz}
If $G$ is a smooth commutative quasi-projective group scheme over $X$, then we can replace the flat site $\mathbf{F}(X)$ in the theorem above with the \'{e}tale site $\mathbf{E}(X)$, see \cite[Theorem III.3.19]{milne1etale}. Specializing further, if $G=\gm\times X$, then we can replace the flat site $\mathbf{F}(X)$ with the Zariski site $\mathbf{Z}(X)$, see \cite[Proposition III.4.9]{milne1etale}.
\end{remark}
By Proposition \ref{asphs} and Theorem \ref{dtiso}, we get the following.
\begin{prop}\label{sesas}
	There exists a short exact sequence (see Example \ref{asexam}(2)):
	\[
	0\ra \mathrm{AS}(G,\si)\ra \mathrm{PHS}_{\sigma}(G/X)\ra \mathrm{PHS}_G(X)^{\sigma_T}\ra 0.
	\]
\end{prop}
\begin{example}\label{isocoho}
Let $(\ka,s)$ be a difference field and $G$ be a group scheme over $\ka$ defined over $\ka^s$. Then we have:
$$H^1_{\sigma}(\ka,G)\cong \check{H}^1_{\sigma}(\mathbf{F}(\spec(\ka)),\mathcal{R}(G)),$$
where $H^1_{\sigma}(\ka,G)$ is the ``difference Galois cohomology'' introduced in \cite{BW}.

In the case when $G$ is commutative, quasi-projective and smooth, we also have:
\begin{IEEEeqnarray*}{rCl}
H^1_{\sigma}(\ka,G) & \cong & H^1_{\sigma}(\mathbf{F}(\spec(\ka)),\mathcal{R}(G)) \\
 &\cong & H^1_{\sigma}(\mathbf{E}(\spec(\ka)),\mathcal{R}(G)) \\
 &\cong & H^1_{\sigma-\mathrm{top}}\left((\gal(\ka),\phi),(\Psi(\mathcal{F}),\Psi(\widetilde{\sigma_\mathcal{F}}))\right),
\end{IEEEeqnarray*}
where the second isomorphism comes from Remark \ref{ftoetoz} and the last one (as well as the description of the last object) from Example \ref{sreptoet}.
Therefore, in the case when $G$ is commutative, quasi-projective and smooth, we have the following short exact sequence:
$$0\ra G(\ka)_{\sigma}\ra H^1_{\sigma}(\ka,G) \ra H^1(\ka,G)^{\sigma}\ra 0,$$
where $H^1(\ka,G)$ is the usual Galois cohomology, and the action of $\sigma$ on the Galois cohomology is induced by the difference structure explained in Section \ref{etaleexam}. In Theorem \ref{wibdesc}, we give a two-fold generalization of the above exact sequence: to an arbitrary $n>0$ (instead of $n=1$), and to an arbitrary difference group scheme $(G,\sigma_G)$ (see Remark \ref{remdifsch}).
\end{example}
The above general example nicely specifies to particular situations, and it allows explicit descriptions of difference torsors. We provide such a description in the case of the difference torsors from Example \ref{exbw1}.
\begin{example}\label{exbw1calc}
We consider the group scheme $\mu_2$ (second roots of unity) over the field $\ka$ such that $\mathrm{char}(\ka)\neq 2$. Using Example \ref{isocoho}, we obtain the following short exact sequence:
$$0\to \mu_2(\ka)_{\sigma}\to H^1_{\sigma}(\ka,\mu_2) \to H^1(\ka,\mu_2)^{\sigma}\to 0.$$
By the classical description of the Galois cohomology of $\mu_2$ (see e.g. \cite[Proposition III.4.11]{milne1etale}), we get:
$$0\to \Zz/2\Zz \to H^1_{\sigma}(\ka,\mu_2) \to \left(\ka^*/(\ka^*)^2\right)^{\sigma}\to 0,$$
which coincides with the description from \cite[Example 5.3]{BW}.
\end{example}
\subsection{Difference bundles and difference Picard group}\label{secdbdpg}
In this subsection, we specialize the situation from Section \ref{contor} further, that is we assume:
$$G=\gl_{n,X}:=X\times \gl_n.$$
We also let $\mathbb{G}_{m,X}$ stand for $\mathrm{GL}_{1,X}$.

We recall several classical facts and definitions. A sheaf of $\mathcal{O}_X$-modules $\mathcal{F}$ is locally trivial in $\mathbf{Z}(X)$ if and only if $\mathcal{F}_{\mathbf{E}}$ is locally trivial in $\mathbf{E}(X)$ if and only if $\mathcal{F}_{\mathbf{F}}$ is locally trivial in $\mathbf{F}(X)$ (see Section \ref{shmodsec} for the notation $\mathcal{F}_{\mathbf{E}},\mathcal{F}_{\mathbf{F}}$). We sometimes call such sheaves \emph{vector bundles} on $X$. A vector bundle is of \emph{rank $n$}, if it is locally isomorphic to $\mathcal{O}_X^n$ in the Zariski topology. We denote the category of vector bundles of rank $n$ by $\mathrm{Vect}^n(X)$. By the associated fiber bundle construction (see e.g. the paragraph just after \cite[Exp. XI Corollaire 4.3]{sga1}), we have the following functor (we denote $\mathrm{Tors}(\gl_{n,X}/X)$ just by $\mathrm{Tors}(\gl_{n,X})$ and similarly for PHS):
$$\mathrm{Ass}: \mathrm{Tors}(\gl_{n,X})\ra \mathrm{Vect}^n(X),$$
which identifies the category of $\gl_{n,X}$-torsors in $\mathbf{Z}(X)$
with the grupoid of isomorphisms of  vector bundles of rank $n$ on $X$.
We denote the set of isomorphism classes of vector bundles of rank $n$ on $X$ by $\mathrm{Bun}^n(X)$. In the special case of $n=1$, we have $\mathrm{Pic}(X)=\mathrm{Bun}^1(X)$ and this set has a natural structure of an abelian group
with respect to the tensor product. Therefore, we have an isomorphism of pointed sets:
\[
\mathrm{PHS}(\gl_{n,X})\cong \mathrm{Bun}^n(X),
\]
and for $n=1$, we have an isomorphism of abelian groups:
\[
\mathrm{PHS}(\mathbb{G}_{m,X})\cong \mathrm{Pic}(X).
\]
We introduce now the difference counterparts. We advice the reader to recall from Section \ref{shmodsec} the notions of left difference sheaves of $\mathcal{O}_X$-modules and the tensor product of them. We also recall that $\si^{\diamond}$ denotes the inverse image in the category of $\mathcal{O}_X$-modules.
\begin{definition}\label{defldb}
A \emph{left difference bundle} is a left difference sheaf $(E,\sigma_E)$ such that $E$ is a locally free sheaf of $\mathcal{O}_X$-modules and we have
$$\widetilde{\si_E}: \si^{\diamond}(E)\cong E.$$
We denote the category of left difference bundles by $\mathrm{Vect}^n_{\sigma}(X)$, and the corresponding set of isomorphism classes by $\mathrm{Bun}^n_{\sigma}(X)$.
\end{definition}
We obtain the following.
\begin{prop}\label{obvrem}
We have the following isomorphism of pointed sets:
$$\mathrm{PHS}_{\sigma}(\gl_{n,X})\cong \mathrm{Bun}_{\sigma}^n(X).$$
\end{prop}
\begin{proof}
Since the Ass-functor takes $\sigma_{\tau}$ to $\sigma^{\diamond}$,
we obtain the result using Definitions \ref{defdgt} and \ref{defldb}.
\end{proof}

We focus now on the (still undefined) difference Picard group. The next result says that invertible left difference sheaves of $\mathcal{O}_X$-modules are the same as difference vector bundles of rank $1$.
\begin{lemma}\label{invdifm}
Let $(\mathcal{F},\sigma_{\mathcal{F}})$ be a left difference sheaf of $\mathcal{O}_X$-module. Then $(\mathcal{F},\sigma_{\mathcal{F}})$ is \emph{invertible}, that is there is a left difference sheaf of $\mathcal{O}_X$-modules $(\mathcal{G},\sigma_{\mathcal{G}})$ such that (see Example \ref{moexam})
$$(\mathcal{F},\sigma_{\mathcal{F}})\otimes (\mathcal{G},\sigma_{\mathcal{G}})\cong (\mathcal{O}_X,s^{\sharp})$$
if and only if the following two conditions hold:
\begin{enumerate}
  \item the sheaf of $\mathcal{O}_X$-modules $\mathcal{F}$ is invertible;
  \item the adjoint map
$$\widetilde{\sigma}_{\mathcal{F}}:\sigma^{\diamond}(\mathcal{F})\ra \mathcal{F}$$
is an isomorphism.
\end{enumerate}
\end{lemma}
\begin{proof}
By the discussion at the beginning of this subsection, we are under the assumption $\mathbf{C}(X)=\mathbf{Z}(X)$ (Zariski topology).

Regarding the implication ``$\Rightarrow$'', it is clear that $\mathcal{F}$ is invertible. To show that the sheaf morphism $\widetilde{\sigma}_{\mathcal{F}}$ is an isomorphism, it is enough to show that it is an isomorphism on stalks. Take $x\in X$ and consider
$$\left(\widetilde{\sigma}_{\mathcal{F}}\right)_x:\sigma^{\diamond}(\mathcal{F})_x\ra \mathcal{F}_x.$$
Since the stalk over $x$ of any invertible sheaf is isomorphic (as an $\mathcal{O}_{X,x}$-module) to $\mathcal{O}_{X,x}$, the map $\left(\widetilde{\sigma}_{\mathcal{F}}\right)_x$ corresponds to a scalar, which we denote by $r_x\in \mathcal{O}_{X,x}$. The map $\left(\widetilde{\sigma}_{\mathcal{F}}\right)_x$ is an isomorphism if and only if $r_x\in (\mathcal{O}_{X,x})^*$. Let $\left(\widetilde{\sigma}_{\mathcal{G}}\right)_x$ correspond to the scalar $s_x\in \mathcal{O}_{X,x}$. Since
$$\widetilde{s^{\sharp}}=\id_{\mathcal{O}_X},$$
and $\id_{\mathcal{O}_{X,x}}$ obviously corresponds to $1_{\mathcal{O}_{X,x}}$, we get that $r_xs_x=1_{\mathcal{O}_{X,x}}$, and $\left(\widetilde{\sigma}_{\mathcal{F}}\right)_x$ is an isomorphism indeed.

For the implication ``$\Leftarrow$'', we recall the inverse of a sheaf $\mathcal{F}$ in the usual Picard group is given
as $\mathrm{Hom}_{\mathcal{O}_X}(\mathcal{F},\mathcal{O}_X)$, hence
$\widetilde{\sigma_{\mathcal{F}}}^{-1}$ induces the structure of a left difference sheaf
on $\mathcal{F}^{-1}$.
\end{proof}
\begin{definition}
We define the \emph{difference Picard group} of $(X,\sigma)$, denoted $\mathrm{Pic}_{\sigma}(X)$, as the set of isomorphism classes of invertible left difference sheaves of $\mathcal{O}_X$-modules with the group operation being the tensor product.
\end{definition}
\begin{remark}\label{rempic}
We note here two facts about the difference Picard group.
\begin{enumerate}
\item By Lemma \ref{invdifm}, we obtain:
$$\mathrm{Pic}_{\sigma}(X)=\mathrm{Bun}^1_{\sigma}(X).$$

\item By Theorem \ref{dtiso} and Remark \ref{ftoetoz}, we get the following isomorphisms of abelian groups:
\begin{IEEEeqnarray*}{rCl}
\mathrm{Pic}_{\sigma}(X) & \cong & H^1_{\sigma}\left(\mathbf{F}(X),\mathcal{R}(\mathbb{G}_{m,X})\right) \\
 &\cong & H^1_{\sigma}\left(\mathbf{E}(X),\mathcal{R}(\mathbb{G}_{m,X})\right) \\
 &\cong & H^1_{\sigma}\left(\mathbf{Z}(X),\mathcal{R}(\mathbb{G}_{m,X})\right).
\end{IEEEeqnarray*}
\end{enumerate}
\end{remark}
Using Remark \ref{rempic}(2), Proposition \ref{sesas} specializes to the following.
\begin{theorem}\label{difpicthm}
We have the following short exact sequence:
\[
0\ra \mathrm{AS}(\mathbb{G}_{m,X},\si)\ra \mathrm{Pic}_{\sigma}(X)\ra \mathrm{Pic}(X)^{\sigma_T}\ra 0.
\]
\end{theorem}
\subsubsection{Difference Picard group in the affine case}\label{secdpaffine}
We consider now the affine case which is a source of interesting examples. Let $(R,s)$ be a commutative difference ring. Then we have the notions of difference modules over $(R,s)$ and tensor products of them  (see Section \ref{shmodsec}). Therefore, we get a natural notion of an invertible difference module, and we obtain the group of their isomorphism classes, denoted $\mathrm{Pic}_{s}(R)$, with the group operation being the tensor product.

If $(M,s_M)$ is a difference module and $M\cong R^n$, then $s_M$ is given by a matrix $A\in M_n(R)$, and we will often write $s_A$ for $s_M$. We note below an easy result (originating from \cite[Fact 3.4]{KP4}), whose proof we leave to the reader.
\begin{lemma}\label{freesmod}
For $A,B\in M_n(R)$, we have:
$$(R^n,\sigma_A)\cong_{(R,\sigma)}(R^n,\sigma_B)\ \ \ \ \ \Longleftrightarrow\ \ \ \ \ \text{there is $C\in \gl_n(R)$ such that }BC=s(C)A.$$
\end{lemma}
We easily obtain below a local version of Proposition \ref{asphs}.
\begin{prop}\label{picrseq}
We have the following exact sequence:
\[1\ra\mathrm{AS}(\gm(R),\gm(s))\ra \mathrm{Pic}_{s}(R)\ra {\rm Pic}(R)^s\ra 1.\]
\end{prop}
\begin{proof}
Using Lemma \ref{freesmod} (to describe the kernel), the argument is the same as in the proof of Proposition \ref{asphs}.
\end{proof}
Similarly as in the classical case, we have a local-global isomorphism of difference Picard groups.
\begin{prop}\label{2picsame}
We have the following isomorphism:
$$\mathrm{Pic}_{s}(R)\cong \mathrm{Pic}_{\sigma}(\spec(R)).$$
\end{prop}
\begin{proof}
By the last comment in Section \ref{shmodsec}, we have a homomorphism
$$\mathrm{Pic}_{s}(R)\ra \mathrm{Pic}_{\sigma}(\spec(R)).$$
Since we have
$$\pic(R)\cong \pic(\spec(R)),\ \ \ \ \mathcal{O}_{\spec(R)}(\spec(R))=R,$$
we get the result by Theorem \ref{difpicthm}, Proposition \ref{picrseq} and the Five Lemma.
\end{proof}
The next result relates the difference Picard group of a difference field with the difference Galois cohomology from Example \ref{isocoho}. As mentioned in the Introduction, this was also our original observation which started this line of research.
\begin{cor}\label{lineres}
If $(\ka,s)$ is a difference field, then we have the following:
$$\gm(\ka)_s\cong \mathrm{Pic}_{s}(\ka)\cong H^1_{\sigma}\left(\ka,\gm\right).$$
\end{cor}
It is clear that in the case of difference number fields our results should have a number-theoretic flavour. We give one example below.
\begin{example}\label{picex}
Let $L/K$ be a cyclic extension of number fields and let $s$ be a generator of the Galois group $G:=\gal(L/K)$. Then, $(\mathcal{O}_L,s)$ is a difference ring and we can compute its difference Picard group (which could be called the \emph{difference class group} of $(K,s)$) $\mathrm{Pic}_{s}(\mathcal{O}_L)$. We clearly have:
$$\mathrm{AS}(\gm(\mathcal{O}_L),s)=L^G\cap \gm(\mathcal{O}_L)=\gm(\mathcal{O}_K),$$
hence the exact sequence from Proposition \ref{picrseq} takes the following form:
\[1\ra\gm(\mathcal{O}_K)\ra \mathrm{Pic}_{s}(\mathcal{O}_L)\ra \mathrm{Cl}(L)^G\ra 1.\]
We would like to point out that the order of the group  $\mathrm{Cl}(L)^G$ (appearing in the short exact sequence above) is classically given by the
 \emph{Chevalley's ambiguous class number formula} (see e.g. \cite[Remark II.6.2.3]{grasclass}).
\end{example}

\section{Right difference sheaves}\label{secrds}
In this section, we develop a theory of right difference sheaves (pairs of the form $(\mcf,\si^{\mcf}: \mcf\to \si^*(\mcf))$, which is,  to some extent, parallel to theory of left difference sheaves from the previous sections. The current situation is much less smooth than the one considered before. The main reason for that seems to be the fact that the functor $\si^*$ does not have a left adjoint, which results in some loss of the flexibility, which we enjoyed in Section \ref{secleftcat}. In particular, we can hardly say anything on presheaves in the present context, since the relation between the right difference sheaves and the right difference presheaves is more complicated and hence less useful than the one in the left difference case.

Our motivation for considering the category of right difference sheaves is the fact that they provide a cohomological description of a much wider class of difference torsors (see Example \ref{mainrex}) than the one from Section \ref{contor}.

We still work under Assumption \ref{mainass}.
\subsection{Definitions and examples}\label{secrde}
In this subsection, we define right difference sheaves and describe our main motivating example.
\begin{definition}\label{defdifsr}
\begin{enumerate}
\item We call a sheaf $\mathcal{F}$ of abelian groups on $\mathbf{C}(X)$ together with a sheaf morphism
$$\sigma^{\mathcal{F}}: \mathcal{F} \ra \sigma^{*}(\mathcal{F}),$$
a {\it right difference sheaf} (of abelian groups on $\mathbf{C}(X)$).
	
\item  A morphism between right difference sheaves $(\mathcal{F},\sigma^{\mathcal{F}})$ and $(\mathcal{G},\sigma^{\mathcal{G}})$ is a sheaf morphism $\alpha:\mathcal{F} \to \mathcal{G}$ such that
    $$\sigma^{\mathcal{G}}\circ\alpha=\sigma^{*}(\alpha)\circ \sigma^{\mathcal{F}}.$$

\item The difference sheaves of abelian groups on $X$ with their morphisms form a category, which we denote $\shxr$.
\end{enumerate}
\end{definition}
Of course, any left difference sheaf $(\mcf,\si_{\mcf})$ for which the adjoint map $\widetilde{\si}_{\mcf}:\si^*({\mcf})\to\mcf$ is invertible naturally gives a right difference sheaf. Thus, for example, the constant sheaf $\Zz_X$ (see Example \ref{constex}) becomes a right difference sheaf as well.

The main source of new examples is provided by sheaves which are represented by differences group schemes. They yield right difference sheaves not coming from left difference sheaves, which is (as was already mentioned above) the main reason why the notion of a right difference sheaf deserves consideration despite all the difficulties.
\begin{remark}\label{remdifsch}
We comment below on a version of the notion of a difference group scheme, which we consider in this paper.
\begin{enumerate}
\item Let $(\ka,s)$ be a difference ring. A \emph{difference group scheme} is defined in \cite{Wib1} as a representable functor from the category of difference $(\ka,s)$-algebras to the category of groups (these objects are actually called ``affine difference algebraic groups'' in \cite{Wib1}). It is clear that a difference group scheme is represented by a pair $(G,\sigma_G)$, where $G$ is an affine group scheme over $\ka$ and $\sigma_G:G\to {}^{\sigma}G$ is a morphism of group schemes over $\ka$.


\item The definition from item $(1)$ above (or, rather, its interpretation from the last sentence there) naturally generalizes to an arbitrary difference base scheme $(X,\sigma)$ instead of the affine one $(\spec(\ka),\spec(s))$, and to arbitrary group schemes instead of the affine ones.  Therefore, when we say a   ``difference group scheme'' in this paper, we mean a pair $(G,\sigma_G)$, where $G$ is a group scheme over $X$ and $\sigma_G:G\to {}^{\sigma}G$ is a morphism of group schemes over $X$.

\item We also have a natural notion of a difference scheme over $(X,\sigma)$, which is a pair $(Y,\sigma_Y)$, where $Y$ is a scheme over $X$ and $\sigma_Y:Y\to {}^{\sigma}Y$ is a morphism of schemes over $X$.
\end{enumerate}
\end{remark}
We describe below our motivating example.
\begin{example}\label{mainrex}
Let $(G,\si_G)$ be a commutative difference group scheme. We assume that $G\in \mathrm{Ob}(\mathcal{C})$. Then $\si_G$ induces the following morphism of sheaves:
\[
\mathcal{R}(\si_G):\mathcal{R}(G)\ra \mathcal{R}({}^{\sigma}G)\cong \sigma^{*}(\mathcal{R}(G)),
\]
where the isomorphism above is $\phi_G^{-1}$ (the morphism $\phi_G$ is from Section \ref{igs}, and it is an isomorphism indeed by \cite[Remark II.3.1.(e)]{milne1etale} and the assumption $G\in \mathrm{Ob}(\mathcal{C})$). Therefore, $\mathcal{R}(G)$ has a natural structure of a right difference sheaf.
\end{example}

\subsection{Properties of the category of right difference sheaves}
Some structures on the category $\shxr$ can be borrowed from the category $\mathrm{Sh}(\mathbf{C}(X))$. Let
\[
o^{\sigma}: \shxr\ra \mathrm{Sh}(\mathbf{C}(X))
\]
be the forgetful functor. This functor was called $i_*$ in the left difference case, since it was induced by a morphism of sites. However, in the present context, we do not have a natural choice of a right difference site, hence we use $o^{\sigma}$ instead of $i_*$ to prevent a possible confusion.
\begin{prop}\label{improve2}
The  category $\shxr$ has the following  properties.
\begin{enumerate}
\item $\shxr$ has all   colimits and finite limits.

\item The functor $o^{\sigma}$  commutes with all   colimits and with finite limits.

\item $\shxr$ is an abelian category.

\item A sequence of difference sheaves
$$0\ra \mathcal{F}\ra \mathcal{G}\ra\mathcal{H}\ra 0$$
is exact if and only if the sequence of sheaves
$$0\ra o^{\sigma}(\mathcal{F})\ra o^{\sigma}(\mathcal{G})\ra o^{\sigma}(\mathcal{H})\ra 0$$
is exact.

\item $\shxr$ satisfies the axiom AB5, that is (in our context): filtered colimits of exact sequences in $\shxr$ are exact.
\end{enumerate}
\end{prop}
\begin{proof}
We recall from Section \ref{secconv} that the functor  $\si^*$ commutes with all colimits and finite limits. Then it follows that for a direct system $(\mc{F}_i,\sigma^{\mc{F}_i})_{i\in I}$ in the category $\shxr$, the pair of the  colimits $({\rm colim}_i
\mc{F}_i,{\rm colim}_i\sigma^{\mc{F}_i})$ in the category $\shxr$  is a colimit
in $\rm{Sh}^{\sigma}(X)$. An analogous argument shows that finite limits exist in  $\rm{Sh}_{\sigma}(X)$. This construction also immediately gives item $(2)$.

For item (3) we observe that a morphism $f$ in $\shxr$ is an isomorphism if and only if $o^{\sigma}(f)$ is. This shows the rest of axioms of abelian categories.

We already know by item $(2)$ that the functor $o^{\sigma}$ is exact. Thus item $(4)$ follows from the fact that $o^{\sigma}$ does not kills objects.

Item $(5)$ follows from the fact that $\mathrm{Sh}(\mathbf{C}(X))$ satisfies the axiom AB5 together with item $(4)$.
 \end{proof}
Therefore, in order to show that  $\shxr$ is a Grothendieck category, we only need to show that it has a generator. To do that, it is convenient to consider a more general situation than the one from Assumption \ref{mainass}, thus we leave Assumption \ref{mainass} for a while. Let $\mathbf{S}=(\mathcal{C},J)$ be a small site
and $\mathcal{F}\in \mathrm{Sh}(\mathbf{S})$. For technical reasons, we define the following notion below, which makes sense in an arbitrary category satisfying the appropriate assumptions. Let $\kappa$ be a cardinal number.
\begin{definition}
We say that the \emph{size of $\mathcal{F}$ is at most $\kappa$}, if for any family $\mathfrak{F}$ of subobjects of $\mathcal{F}$ such that $\mathfrak{F}$ covers $\mathcal{F}$ (i.e. $\sum \mathfrak{F} = \mathcal{F}$), there is a subfamily $\mathfrak{F}_0\subseteq \mathfrak{F}$ such that $|\mathfrak{F}_0|\leqslant \kappa$, and $\mathfrak{F}_0$ still covers $\mathcal{F}$.
\end{definition}
We have the following obvious observation.
\begin{lemma}\label{bound}
We assume that:
\begin{itemize}
\item $(\mathcal{F}_i)_{i\leqslant \alpha}$ is a family of sheaves such that each $\mathcal{F}_i$ is of size at most $\kappa$;

\item $\alpha\leqslant \kappa$;

\item $\mathcal{H}$ is a subsheaf of $\bigoplus_{i<\alpha}\mathcal{F}_i$.
\end{itemize}
Then the size of the quotient sheaf $(\bigoplus_{i<\alpha}\mathcal{F}_i)/\mathcal{H}$ is at most $\kappa$.
\end{lemma}
We fix $\mathcal{G}\in \mathrm{Sh}(\mathbf{S})$, which is a generator for $\mathrm{Sh}(\mathbf{S})$, and we also fix an infinite cardinal number $\kappa$ such that the size of $\mathcal{G}$ is at most $\kappa$. Let $\mathcal{G}^{(\kappa)}$ stand for $\bigoplus_{i<\kappa}\mathcal{G}$.

Let $\sigma:\mathbf{S}\to \mathbf{S}$ be a morphism of sites. There is an obvious notion of a right difference sheaf in this context and right difference sheaves form a category, which we denote by  $\mathrm{Sh}^{\sigma}(\mathbf{S})$.
\begin{theorem}\label{generators}
Let $\mathfrak{G}$ be the family consisting of $(\mathcal{F},\sigma^{\mathcal{F}})\in \mathrm{Sh}^{\sigma}(\mathbf{S})$ such that the sheaf $\mathcal{F}$ is a quotient of $\mathcal{G}^{(\kappa)}$. Then $\mathfrak{G}$ is a generating family for the category $\mathrm{Sh}^{\sigma}(\mathbf{S})$.
\end{theorem}
\begin{proof}
Let $(\mathcal{H},\sigma^{\mathcal{H}})\in \mathrm{Sh}^{\sigma}(\mathbf{S})$. There is a family of sheaf morphisms:
$$\{f_j:\mathcal{G}\ra \mathcal{H}\ |\ j\in J\}$$
such that the sheaf $\mathcal{H}$ is generated by the family $(\mathcal{A}_j:=\im(f_j))_{j\in J}$. For each $j\in J$, it is enough to find  $(\mathcal{F},\sigma^{\mathcal{F}})\in \mathfrak{F}$ such that $(\mathcal{F},\sigma^{\mathcal{F}})$ is a \emph{difference subsheaf} of $\mathcal{H}$ (which means that $\mathcal{F}$ is a subsheaf of $\mathcal{H}$ and $\sigma^\mathcal{F}=\sigma^\mathcal{H}|_{\sigma^*({\mathcal{F}})}$) and $\mathcal{F}$ contains $\mathcal{A}_j$. Let us fix $j_0\in J$ and denote $\mathcal{A}^0:=\mathcal{A}_{j_0}$.

Since the functor $\sigma^*$ commutes with direct limits, the family $\sigma^*(\mathcal{A}_j)_{j\in J}$ covers the sheaf $\sigma^*(\mathcal{H})$. By Lemma \ref{bound}, the size of $\sigma^{\mathcal{G}}(\mathcal{A}^0)$ is at most $\kappa$. Hence there is $J_0\subseteq J$ such that $|J_0|\leqslant \kappa$ and $\sigma^{\mathcal{G}}(\mathcal{A}^0)$ is covered by the family $\sigma^*(\mathcal{A}_j)_{j\in J_0}$. Let us define:
$$\mathcal{A}^1:=\left\langle \mathcal{A}^0\cup \bigcup_{j\in J_0} \mathcal{A}_j \right\rangle .$$
Then we have the following:
\begin{itemize}
\item $\mathcal{A}^0\leqslant \mathcal{A}^1\leqslant \mathcal{H}$;

\item $\sigma^{\mathcal{G}}(\mathcal{A}^0)\leqslant \sigma^{*}(\mathcal{A}^1)$;

\item the size of $\mathcal{A}^1$ is at most $\kappa$ (by Lemma \ref{bound}).
\end{itemize}
Continuing like this, we obtain the desired difference subsheaf $\mathcal{G}:=\bigcup_{i<\omega}\mathcal{A}^i$, where $\sigma^\mathcal{G}:=\sigma^\mathcal{H}|_{\sigma^*({\mathcal{G}})}$.
\end{proof}
By Proposition \ref{improve2}(5) and Theorem \ref{generators}, we get the following.
\begin{cor}
$\shxr$ is a Grothendieck category. In particular, it has all limits and enough injectives.
\end{cor}
We have shown that $\shxr$ is a reasonable abelian category. However, the reader should be aware that some parts of its structure are quite elusive. For example, although infinite products exist in $\shxr$ (since it is a Grothendieck category), the forgetful functor
$$o^{\sigma}:\shxr\to \sx$$
need not preserve them (and it probably does not, at least we do not see a reasonable right difference structure on the infinite product of ordinary sheaves). By the Special Adjoint Functor Theorem, the functor $o^{\sigma}$ has a right adjoint, but it is not clear how to obtain it by any simple construction, which was the case for the left difference sheaves. Therefore, we were not able to  obtain a generator for the category $\shxr$ by simply dragging a generator from  the category  $\mathrm{Sh}(\mathbf{C}(X))$.

\section{Right difference cohomology}\label{secrdcoh}
In this section, we develop a theory of right difference sheaf cohomology. The main results we obtain are right difference sheaf versions of the results from Section \ref{secleftdc}, but the methods and the proofs are often different that the ones from Section \ref{secleftdc}.

For the same reasons as in Section  \ref{secleftdc}, we often remove $\mathbf{C}$ from our notation and, for example, we abbreviate $\shxr$ just to $\rx$.
\subsection{Right difference sheaf cohomology}
We recall from Section \ref{secrde} that the constant sheaf $\Zz_X$ has a natural structure of a right difference sheaf. We define the {\it right difference global section functor} (denoted by the same symbol as in Section \ref{secldifcoh}) as:
$$\Gamma^{\sigma}(\mcf):=\mathrm{Hom}_{\rx}(\Zz_X,\mcf).$$
Explicitly, we have:
$\Gamma^{\sigma}(\mcf)=\ker(\mathrm{res}-\si^{\mcf})$ for the evident maps:
\[
\mathrm{res}, \si^{\mcf}:\Gamma(\mcf)\ra\Gamma(\si^*(\mcf)).
\]
\begin{definition}
For $\mc{F}\in \rx$, we define the \emph{$n$-th right difference sheaf cohomology} (still denoted by the same symbol as in Section \ref{secldifcoh}):
\[H^n_{\sigma}(X,\mc{F}):=R^n\Gamma^{\sigma}(X)={\rm Ext}^n_{\rx}(\Zz_X,\mc{F}).
\]
\end{definition}
We shall factorize the functor $\Gamma^{\sigma}$ in order to obtain the short exact sequence corresponding to the one from Theorem \ref{mainss}. Unfortunately,  in the present context the map $\si^{\mcf}$ does not act on $\Gamma(\mcf)$, therefore we also need to modify the intermediate category used in our factorization.

Let $\mathcal{Q}$ stands for the category of integral representations of the quiver  $\mathcal{A}_2$ (see \cite[Definition 1.2 and Example 1.4]{quiver}).
Explicitly, the objects of $\mathcal{Q}$ are triples $(Q_0, Q_1,x)$ such that $Q_0, Q_1$ are abelian groups and $x:Q_0\to Q_1$ is a homomorphism.
We define the functor
$$\gamma=(\gamma_0,\gamma_1,x): \rx\to\mathcal{Q}$$
by the following formula:
$$\gamma_0\left(\mcf,\si^{\mcf}\right):=\Gamma(\mcf),\ \ \ \gamma_1\left(\mcf,\si^{\mcf}\right):=\Gamma(\si^*(\mcf)),\ \ \  x\left(\mcf,\si^{\mcf}\right)=\mathrm{res}-\si^{\mcf}.$$
We also define the functor:
$$k:\mathcal{Q}\to \mathbf{Ab},\ \ \ \ \ \ k(Q_0, Q_1,x)=\ker(x).$$
Then we have a natural isomorphism $\Gamma^{\sigma}\cong k\circ\gamma$, which is the factorization we were looking for. We should show that $\gamma$ takes injective objects to $k$-acyclic objects. To this end, we observe the following first.
\begin{lemma}\label{xepic}
An object $(Q_1,Q_2,x)\in\mathcal{Q}$ is $k$-acyclic if and only if $x$ is epic.
\end{lemma}
\begin{proof} Let $\Zz_0, \Zz_1, \Lambda(x)\in\mathcal{Q}$ be such that:
$$\Zz_0:=(\Zz,0,0),\ \ \ \Zz_1:=(0,\Zz,0),\ \ \ \Lambda(x)=(\Zz,\Zz,\id).$$
We observe that $\Zz_0$ represents the functor $k$, that is:
$$k(Q)\cong \Hom_{\mathcal{Q}}(\Zz_0,Q)$$
for any $Q\in \mathcal{Q}$. Hence, for each $Q=(Q_1,Q_2,x)\in\mathcal{Q}$, we need to show that $\mathrm{Ext}^{>0}_{\mathcal{Q}}(\Zz_0,Q)=0$ if and only if $\ker(x)=0$.

The following obvious short exact sequence in $\mathcal{Q}$:
\[ 0\ra \Zz_1\ra\Lambda(x)\ra\Zz_0\ra 0
\]
is the standard resolution of $\Zz_0$ (see \cite[Section 1.4]{quiver}); in particular, it is a projective resolution. Applying the functor $\Hom_{\mathcal{Q}}(-,Q)$ to this last exact sequence, we obtain the complex
\[0\ra Q_0\stackrel{x}{\ra} Q_1\ra 0\]
computing $\mathrm{Ext}^*_{\mathcal{Q}}(\Zz_0,Q)$. Hence, the result follows.
\end{proof}
We can conclude now the result below, which will be crucial in the sequel.
\begin{lemma}\label{embkacy}
Any $\mcf\in\rx$ embeds into $\mathcal{G}\in\rx$ such that $\gamma(\mathcal{G})$ is $k$-acyclic.
\end{lemma}
\begin{proof}
We take $\mcf\in\rx$ and define $\mathcal{G}:=\bigoplus_{j\geqslant 0} (\si^j)^*(\mcf)$.
Then we have $\si^*(\mathcal{G})\cong \bigoplus_{j\geqslant 1} (\si^j)^*(\mcf)$, and we set:
$$\si^{\mathcal{G}}:\mathcal{G}\ra \si^*(\mathcal{G}),\ \ \ \si^{\mathcal{G}}(a_0,a_1,a_2,\ldots)=(a_1,a_2,\ldots)$$
(the reader may compare this construction with its ``left variant'' from the proof of Proposition \ref{improve}).

By Proposition \ref{improve2}(4), the obvious sheaf embedding $\mathcal{F}\to \mathcal{G}$ (on the $0$-th summand) is a monomorphism in $\rx$.
The fact that  $\mathrm{res}-\si^{\mathcal{G}}$ is  onto on global sections follows from the  elementary fact true for graded bounded below abelian groups, which is: if
$$f_0, f_1: \bigoplus_{n=0}^{\infty}A_n\ra \bigoplus_{n=0}^{\infty}B_n$$
are homomorphisms, $f_i$ is of degree $i$ for $i=0,1$,  and $f_0$ is epimorphic; then
$f_0+f_1$ is epimorphic as well. In our case, we take:
$$A_i=B_i:=\Gamma((\si^j)^*(\mcf)),\ \ f_0=\id,\ \ f_1=\mathrm{res}.$$
Then,  Lemma \ref{xepic}  shows that $\gamma(\mathcal{G})$ is $k$-acyclic.
\end{proof}
We derive from the lemmas above the result we need.
\begin{prop}\label{rgrok}
The functor $\gamma$ takes the injective objects from the category $\rx$ to the $k$-acyclic ones in the category $\mathcal{Q}$.
\end{prop}
\begin{proof} Let $\mcf\in\rx$ be injective. By Lemma \ref{embkacy}, there is a monomorphism $\mcf \to \mc{G}$ in the category $\rx$ such that $\gamma(\mc{G})$ is $k$-acyclic. We extend this monomorphism to a short exact sequence $0\to \mathcal{F}\to \mathcal{G}\to \mathcal{H}\to 0$, and consider the following commutative diagram:
\begin{equation*}
\xymatrix{ & & 0 & & \\
0  \ar[r]^{}  & \Gamma(\si^*(\mcf)) \ar[r]^{} &  \Gamma(\si^*(\mc{G})) \ar[r]^{} \ar[u]^{} & \Gamma(\si^*(\mc{H}))  &  \\
0  \ar[r]^{}  & \Gamma(\mcf) \ar[r]^{} \ar[u]^{?} & \Gamma(\mc{G}) \ar[r]^{} \ar[u]^{} & \Gamma(\mc{H}) \ar[u]^{}  &  \\
0  \ar[r]^{}  & \Gamma^{\sigma}(\mcf) \ar[r]^{} \ar[u]^{} & \Gamma^{\sigma}(\mc{G}) \ar[r]^{} \ar[u]^{} & \Gamma^{\sigma}(\mc{H}) \ar[r]^{} \ar[u]^{}  & 0 \\
  & 0  \ar[u]^{} & 0 \ar[u]^{} & 0, \ar[u]^{}  & }
\end{equation*}	
where the arrows from $\Gamma(-)$ to $\Gamma\circ\sigma^*(-)$ are of the form $\mathrm{res}-\sigma^{(-)}$.
We claim that the rows and columns in this diagram are exact. Indeed, the exactness of the columns follows from the definition of $\Gamma^{\sigma}$ and the $k$-acyclicity of $\gamma(\mc{G})$, while   the exactness of the lower row follows from the injectivity of ${\mcf}$. The exactness of the middle row follows from the left exactness of $\Gamma$, and the exactness of the upper row is given by the exactness of $\si^*$.

It is enough to show the epimorphicity of the vertical arrow with the question mark inserted, which follows from a diagram chasing (it is a special case of the Snake Lemma).
\end{proof}
Analogously to Section \ref{secleftdc}, we will describe now the right difference cohomology as the cohomology of a certain bicomplex. Let $(\mcf, \sigma^{\mcf})\in\rx$  and  $\iota:\mcf\to I^{*}$ be  its injective resolution in $\rx$. Let further $\si^*(\mcf)\to J^*$ be an injective resolution of $\si^*(\mcf)$. Since the functor $\sigma^*$ is exact (see Section \ref{secconv}), by the injectivity of $J^*$ we get a cochain morphism $c: \sigma^*(I^*)\ra J^*$ such that the following diagram commutes:
\begin{equation*}
\xymatrix{\si^*(\mcf) \ar[rd]^{} \ar[r]^{\sigma^*(\iota)} & \si^*(I^*) \ar[d]^{c}\\
 &  J^*.}
\end{equation*}
Then, we have the following map
\[ \bar{\si}^{\mcf}:=c\circ \si^{I^*}:\Gamma(I^*)\ra\Gamma(J^*)\]
and we clearly get:
$$H^*\left(\bar{\si}^{\mcf}\right)=H^*\left(X,\si^{\mcf}\right):H^*(X,\mathcal{F})\ra H^*(X,\mathcal{G})$$
on the sheaf cohomology. On the other hand, we have the following composition map:
\begin{equation*}
\xymatrix{\Gamma(I^*) \ar[r]^{\mathrm{res}_{I^*}\ \ \ \ } & \Gamma(\sigma^*(I^*)) \ar[r]^{\Gamma(c)} &  \Gamma(J^*),}
\end{equation*}
which on the sheaf cohomology induces the same map as the map induced by the restriction morphism $\mathrm{res}:\Gamma(\mcf)\to\Gamma(\si^*(\mcf))$ considered above.
\\
We define the bicomplex $D^{**}(\mathcal{F})$  with two non-trivial rows:
$\Gamma(I^*)$ and $\Gamma(J^*)$, with the vertical (upward) differential $\pm(\mathrm{res}-\bar{\sigma}^{\mcf})$, and with the horizontal differentials induced by the maps from the chosen injective resolutions. Then, as in Section \ref{secldifcoh}, we have the following.
\begin{prop}\label{dcpxr}
 There is a natural in $\mathcal{F}$ isomorphism:
\[ H^*_{\sigma}(X,\mathcal{F})\cong H^*({\rm Tot}(D(\mathcal{F}))).\]	
\end{prop}
\begin{proof} It is an obvious modification of the proof of Proposition \ref{dcpxl}.
The crucial ingredient is Proposition \ref{rgrok} (we also use the exactness of $\sigma^*$).
\end{proof}
In the following theorem,
$H^n(X,\mathcal{F})^{\sigma}$ (resp. $H^n(X,\si^*(\mathcal{F}))_{\sigma}$) stands
 for the kernel (resp. cokernel) of the map:
 $$  H^n(X,\mathcal{F})\ra H^n(X,\si^*(\mathcal{F})),$$
 which is  induced on cohomology by the map $\mathrm{res}-\si^{\mcf}$.
\begin{theorem}\label{mainss2}
For any $n\geqslant  0$, there is a short exact sequence:
	\[0\ra H^{n-1}(X,\si^*(\mathcal{F}))_{\sigma}\ra H^n_{\sigma}(X,\mathcal{F})\ra H^n(X,\mathcal{F})^{\sigma}\ra 0,\]
where $H^{-1}(X,\mathcal{F}):=0$.
\end{theorem}
\begin{proof}
The result follows from the first spectral sequence associated to the bicomplex $D^{**}(\mcf)$.
\end{proof}

\subsection{Right difference \v Cech cohomology}\label{rightcechsec}
Unlike in Section \ref{leftcechsec}, we consider here \v Cech cohomology only for right difference \emph{sheaves} (and not for right difference presheaves).
Let $\mcf\in\rx$ and let $\mathcal{U}$ be a covering of $X$. We have the following two composition maps:
\begin{equation*}
\xymatrix{\check{C}^*(\mathcal{U},\mcf) \ar[rr]^{\check{C}^*\left(\sigma^{\mathcal{F}}\right)\ \ \ \ } & & \check{C}^*(\mathcal{U},\si^*(\mcf)) \ar[rr]^{\mathrm{res}\ \ \ \ } & &  \check{C}^*(\mathcal{U}\times_X {}^{\sigma}\mathcal{U},\si^*(\mcf)),}
\end{equation*}
\begin{equation*}
\xymatrix{\check{C}^*(\mathcal{U},\mcf) \ar[rr]^{\check{\sigma}\ \ } & & \check{C}^*({}^{\sigma}\mathcal{U},\si^*(\mcf)) \ar[rr]^{\mathrm{res}\ \ \ \ } & &  \check{C}^*(\mathcal{U}\times_X {}^{\sigma}\mathcal{U},\si^*(\mcf));}
\end{equation*}
where $\check{\sigma}$ is the map induced by the fixed scheme morphism $\sigma:X\to X$. We denote the first composition above by $\check{\sigma}^{\mcf}$ and the second one by $\mathrm{res}_{\sigma}$.

Analogously to Section 3.2, we define the bicomplex
$\check{C}^{**}_{\sigma}(\mathcal{U},\mc{F})$ by putting:
$$\check{C}^{*0}_{\sigma}(\mathcal{U},\mc{F})=\check{C}^{*}(\mathcal{U},\mc{F}),$$
$$\check{C}^{*1}_{\sigma}(\mathcal{U},\mc{F})
=\check{C}^{*}(\mathcal{U},\si^*(\mc{F}))$$
and trivial elsewhere.
The horizontal differential comes from the standard \v{C}ech complex and the vertical differential is the map $\pm(\mathrm{res}_{\sigma}-\check{\sigma}^{\mcf})$. We call this bicomplex the \emph{right difference \v Cech bicomplex} and we call its cohomology by the \emph{right difference \v Cech cohomology}, which we denote by $\check{H}^*_{\sigma}(\mc{U},\mcf)$.

In the present context we make no claims on the properties of \v Cech cohomology except those following directly from the properties of the defining bicomplex $\check{C}^{**}_{\sigma}(\mathcal{U},\mc{F})$. In particular, we have the following.
\begin{theorem}\label{cechssp}
For any $n\geqslant  0$,
there is a short exact sequence
	\[0\ra \check{H}^{n-1}(\mathcal{U},\mathcal{F})_{\sigma}\ra \check{H}^n_{\sigma}(\mathcal{U},\mathcal{F})\ra \check{H}^n(\mathcal{U},\mathcal{F})^{\sigma}\ra 0,\]
 where $\check{H}^{-1}(\mathcal{U},\mathcal{F}):=0$.
\end{theorem}
\begin{proof} It follows immediately
from the first spectral sequence for the difference \v Cech bicomplex.
\end{proof}
The proof of the next result follows the lines of the proof of Theorem \ref{cechtoder}, so we skip it.
\begin{theorem}\label{cechtoder1r}
We have  a natural map:
$$\alpha:\check{H}_{\sigma}^p(\mathcal{U},\mathcal{F})\ra H_{\sigma}^p(X,\mathcal{F}),$$
which  fits into the following commutative diagram:
\begin{equation*}
\xymatrix{0  \ar[r]^{}  & \check{H}^{n-1}(\mathcal{U},\si^*(\mathcal{F}))_{\sigma} \ar[r]^{} \ar[d]^{\beta_{\sigma}} & \check{H}^n_{\sigma}(\mathcal{U},\mathcal{F}) \ar[r]^{} \ar[d]^{\alpha} & \check{H}^n(\mathcal{U},\mathcal{F})^{\sigma} \ar[r]^{} \ar[d]^{\beta^{\sigma}} & 0 \\
0  \ar[r]^{}  & H^{n-1}(X,\si^*(\mathcal{F}))_{\sigma} \ar[r]^{} & H^n_{\sigma}(X,\mathcal{F}) \ar[r]^{} &  H^n(X,\mathcal{F})^{\sigma} \ar[r]^{} & 0 ,}
\end{equation*}
where $\beta$ is the usual  map from \v Cech to sheaf cohomology \cite[(I.3.4.5)]{tamme}. \\
Therefore, if $\beta$ is an isomorphism for  some $\mcf$ in degrees: $n, n-1$, then $\alpha$ is an isomorphism for this $\mcf$ in degree $n$.
\end{theorem}
Analogously to Section \ref{leftcechsec} again, we define the limit  difference \v Cech cohomology:
\[
\check{H}^n_{\sigma}(X,\mcf):=
\coli_{\mathcal{U}}\check{H}^n_{\sigma}(\mc{U},\mcf).
\]
(where the limit runs through the family of  coverings of $X$), and we have the following limit counterparts of Theorems \ref{cechssp} and \ref{cechtoder1r}.
\begin{theorem}\label{cechss1}
For any $n\geqslant  0$,
there is a short exact sequence
	\[0\ra \check{H}^{n-1}(X,\sigma^*(\mathcal{F}))_{\sigma}\ra \check{H}^n_{\sigma}(X,\mathcal{F})\ra \check{H}^n(X,\mathcal{F})^{\sigma}\ra 0,\]
 where $\check{H}^{-1}(X,\mathcal{F}):=0$.
\end{theorem}
\begin{theorem}\label{cechtoder2r}
We have  a natural map:
$$\alpha:\check{H}_{\sigma}^p(X,\mathcal{F})\ra H_{\sigma}^p(X,\mathcal{F}),$$
which  fits into the  commutative diagram:
\begin{equation*}
\xymatrix{0  \ar[r]^{}  & \check{H}^{n-1}(X,\si^*(\mathcal{F}))_{\sigma} \ar[r]^{} \ar[d]^{\beta_{\sigma}} & \check{H}^n_{\sigma}(X,\mathcal{F}) \ar[r]^{} \ar[d]^{\alpha} & \check{H}^n(X,\mathcal{F})^{\sigma} \ar[r]^{} \ar[d]^{\beta^{\sigma}} & 0 \\
0  \ar[r]^{}  & H^{n-1}(X,\si^*(\mathcal{F}))_{\sigma} \ar[r]^{} & H^n_{\sigma}(X,\mathcal{F}) \ar[r]^{} &  H^n(X,\mathcal{F})^{\sigma} \ar[r]^{} & 0 ,}
\end{equation*}
where $\beta$ is the usual  map from \v Cech to sheaf cohomology \cite[(I.3.4.5)]{tamme}.\\
Therefore, if $\beta$ is an isomorphism for  some $\mcf$ in degrees: $n, n-1$, then $\alpha$ is an isomorphism for this $\mcf$ in degree $n$.\\
In particular, $\alpha$ is
always an isomorphism for $n=0,1$, or for any $n$ when
 $\mathbf{C}(X)=\mathbf{Z}(X)$, $X$ is a separated scheme, and $\mathcal{F}$ is a  difference quasi-coherent sheaf.
\end{theorem}

\section{Right difference sheaf torsors}\label{rdstsec}
Let $\mc{G}$ be a sheaf of groups on $\mathbf{C}(X)$. We recall from Section \ref{ldstsec} the category $\mathrm{TSh}(\mc{G}/\mathbf{F}(X))$, and the functors
\[
f^*:\mathrm{TSh}(\mc{G}/Y)\ra \mathrm{TSh}(f^*(\mc{G})/X),\ \ \ \
\alpha_*:\mathrm{TSh}(\mc{G}/X)\ra \mathrm{TSh}(\mc{H}/X);
\]
where $\alpha:\mc{G}\to \mc{H}$ is a morphism of group sheaves on $\mathbf{F}(X)$, and $f:X\to Y$ is a morphism of schemes. Let $\mc{P}$ be a sheaf of $\mc{G}$-sets on $\mathbf{C}(X)$. Then $\sigma^*(\mc{P})$ is a sheaf of $\sigma^*(\mc{G})$-sets.
\begin{definition}\label{defrtor}
We assume that $(\mc{G},\sigma^{\mc{G}})$ is a right difference sheaf of groups (not necessarily abelian groups!).
\begin{enumerate}
\item A \emph{right difference sheaf $\mc{G}$-torsor} is a pair $(\mc{P},\sigma^{\mc{P}})$, where $\mc{P}\in \mathrm{TSh}(\mc{G}/\mathbf{F}(X))$
and
\[
\sigma^{\mc{P}}: (\sigma^{\mc{G}})_*(\mc{P})\ra \sigma^*(\mc{P})
\]
is an isomorphism of sheaf $\sigma^*(\mc{G})$-torsors.

\item The right difference sheaf $\mc{G}$-torsors on $\mathbf{F}(X)$ form a category,
and we denote the set of isomorphism classes of right difference sheaf $\mc{G}$-torsors on $\mathbf{F}(X)$ by $\mathrm{PHSh}^{\sigma}(\mc{G}/\mathbf{F}(X))$.
\end{enumerate}
\end{definition}
We obtain below a difference version of the classical correspondence between torsors and the first cohomology group. Its proof is entirely analogous to the proof of Theorem \ref{lphsthm}, so we skip it.
\\
By adjusting the constructions from Section 4.1 to the context of right difference torsors we
obtain  a difference version of the classical correspondence between sheaf torsors and the first cohomology group. We define the notion of the first pointed-set right difference cohomology with non-commutative coefficients in an analogous way as in Definition \ref{defcnon}.
\begin{theorem}\label{rphsthm}
There is an isomorphism of pointed sets, which is an isomorphism of abelian groups when $\mc{G}$ is commutative:
\[
\mathrm{PHSh}^{\sigma}(\mc{G}/X)\cong \check{H}^1_{\sigma}(X,\mc{G}).
\]
\end{theorem}

\subsection{Right difference torsors}\label{secrdt}
In this subsection, we are finally arriving to our main motivation for introducing the notion of a right difference sheaf. We make the same assumptions here as in Section \ref{contor}, that is: $\mathbf{C}(X)=\mathbf{F}(X)$, $G$ is a flat group scheme over $X$ satisfying one of the assumptions of \cite[Theorem III.4.3]{milne1etale}, which guarantee that any sheaf of $\mathcal{R}(G)$-torsors on $\mathbf{F}(X)$ come from a ``usual $G$-torsor.''

We recall the following natural definition (\cite[Definition 1.2]{BW}). It is phrased in \cite{BW} in terms of representable functors, however we give an equivalent definition which fits into our terminology of schemes and morphisms between them.
\begin{definition}\label{dtordef}
Let $(G,\sigma_G)$ be a difference group scheme over $(X,\sigma)$ as in Remark \ref{remdifsch}(2).
\begin{enumerate}
\item A \emph{difference $G$-torsor} on $X$ is a $G$-torsor $P$ together with a structure $(P,\sigma_P)$ of a difference scheme (see Remark \ref{remdifsch}(3)) over $(X,\sigma)$ such that the $G$-torsor action $\mu:G\times_XP\to P$ is a morphism of difference schemes over $(X,\sigma)$, that is, the following diagram of $X$-scheme morphisms commutes:
\begin{equation*}
\xymatrix{G\times_XP  \ar[rr]^{\mu} \ar[d]_{\sigma_G\times \sigma_P}  & & P  \ar[d]_{\sigma_P} \\
{}^{\sigma}G\times_X{}^{\sigma}P \ar[rr]^{^{\sigma}\mu} & & ^{\sigma}P.}
\end{equation*}

\item There is a natural notion of a morphism of difference torsors, and such torsors form a category. We denote the set of isomorphism classes of difference torsors over $(G,\sigma_G)$ by $\mathrm{PHS}^{\sigma}(G/X)$.
\end{enumerate}
\end{definition}
\begin{remark}\label{anotdt}
\begin{enumerate}
\item It is easy to see that a difference $G$-torsor on $X$ is the same as  a $G$-torsor $P$ together with the choice of a structure $(P,\sigma_P)$ of a difference scheme over $(X,\sigma)$ such that $\sigma_P$ gives rise to the isomorphism of ${}^{\sigma}G$-torsors (denoted by the same symbol):
\[
\sigma_P: (\sigma_G)_*(P)\cong {}^{\sigma}P.\]

\item Let $(P,\sigma_P)$ be a difference $G$-torsor on $X$. Then $P$ is trivial as a non-difference torsor if and only if $P(X)\neq \emptyset$ and, moreover, the set $P(X)$ corresponds to the set of $G$-torsor isomorphisms between $P$ and $G$ (see e.g. \cite[page 120]{milne1etale}). Let us now consider the corresponding functor of ``$\sigma$-rational points'' (see Remark \ref{remdifsch}), which we denote by $P^{\sharp}$. Then $P^{\sharp}((X,\sigma))$ is a subset of $P(X)$ corresponding to those isomorphisms of $G$-torsors, which also preserve the difference structure. In particular, analogously to the classical context, $(P,\sigma_P)$ is a trivial difference $G$-torsor on $X$ if and only if $P^{\sharp}((X,\sigma))\neq \emptyset$.
\end{enumerate}
\end{remark}
\begin{example}\label{exbw2}
We give one more example from \cite{BW} rephrased in our terminology (see Example \ref{exbw1}). Let $X=\spec(\ka)$, where $\ka$ is a field. We fix $\lambda_0,\lambda_1,\ldots,\lambda_{n-1}\in \ka$, and consider the following ``Picard-Vessiot'' difference group scheme:
$$\left(\ga^n,\sigma_{\ga^n}\right),\ \ \ \ \sigma_{\ga^n}(a_1,\ldots,a_{n-1},a_n)=(a_2,\ldots,a_{n},\lambda_0a_1+\ldots+\lambda_{n-1}a_{n}).$$
Since $\ga^n$ is defined over constants of $s$, we can assume that ${}^{\sigma}\ga^n=\ga^n$ (see Section \ref{igs}). For any $\lambda\in \ka$, we define the following difference $\ga^n$-torsor:
$$\left(\ga^n,\sigma_{\lambda}\right),\ \ \ \ \sigma_{\lambda}(a_1,\ldots,a_{n-1},a_n)=(a_2,\ldots,a_{n},\lambda_0a_1+\ldots+\lambda_{n-1}a_{n}+\lambda).$$
Then the group operation morphism on $\ga^n$ gives $(\ga^n,\sigma_{\lambda})$ the structure of a difference $\ga^n$-torsor as in \cite[Example 1.4]{BW}.
\end{example}
For a difference group scheme $(G,\sigma_G)$, we would like to classify the difference $G$-torsors by applying Theorem \ref{rphsthm}.
By Example \ref{mainrex}, $\mathcal{R}(G)$ is a right difference sheaf of possibly non-commutative groups.
\begin{theorem}\label{rdiso}
There is the following isomorphism of pointed sets:
$$\mathrm{PHS}^{\sigma}(G/X)\cong \mathrm{PHSh}^{\sigma}({\mathcal{R}(G)}/\mathbf{F}(X)),$$
which is an isomorphism of abelian groups in the case when $G$ is commutative.
\end{theorem}
\begin{proof}
Let $(P,\sigma_P)$ be a difference $G$-torsor on $X$. By Remark \ref{anotdt}, $\sigma_P$ gives rise
to the isomorphism of ${}^{\sigma}G$-torsors:
\[
\sigma_P: (\sigma_G)_*(P)\ra {}^{\sigma}P.\]
After applying the ``representability functor'' $\mathcal{R}$, we get the following isomorphism:
$$\mathcal{R}(\sigma_G)_*(\mathcal{R}(P))=\mathcal{R}((\sigma_G)_*(P))\ra \mathcal{R}({}^{\sigma}P),$$
where the equality follows in the same way as in the first displayed line in the proof of Lemma \ref{torbij}. We apply the functor $(\phi_G^{-1})_*$ to this last isomorphism, and obtain the following isomorphism:
$$(\phi_G^{-1})_*\circ \mathcal{R}(\sigma_G)_*(\mathcal{R}(P))\ra (\phi_G^{-1})_*\circ \mathcal{R}({}^{\sigma}P)=\sigma^*(\mathcal{R}(P)),$$
where the equality comes from the upper of the commutative diagram appearing at the end of the proof of Lemma \ref{torbij}. Using Example \ref{mainrex} (for the second equality below), we obtain the following:
$$(\phi_G^{-1})_*\circ \mathcal{R}(\sigma_G)_*=\left(\phi_G^{-1}\circ \mathcal{R}(\sigma_G)\right)_*=\left(\sigma^{\mathcal{R}(G)}\right)_*.$$
Hence, we get an isomorphism:
$$\left(\sigma^{\mathcal{R}(G)}\right)_*(\mathcal{R}(P))\ra \sigma^*(\mathcal{R}(P)),$$
which gives $\mathcal{R}(P)$ the structure of a right difference sheaf $\mathcal{R}(G)$-torsor.
\end{proof}
\begin{cor}\label{fdc}
Using Theorems \ref{rphsthm} and \ref{rdiso}, we get the following isomorphism:
$$\mathrm{PHS}^{\sigma}(G/X)\cong \check{H}^1_{\sigma}(\mathbf{F}(X),\mathcal{R}(G)).$$
If $X=\spec(\ka)$, where $\ka$ is a field, then we get:
$$H^1_{\sigma}(\ka,G)\cong \check{H}^1_{\sigma}(\mathbf{F}(\spec(\ka)),\mathcal{R}(G)),$$
where the group $H^1_{\sigma}(\ka,G)$ was introduced in \cite{BW}.
\end{cor}
\begin{remark}
The above result looks similar to Example \ref{isocoho}, but one should note that only the case of isotrivial difference group schemes (i.e. of the form $(G,\id)$ for $G$ defined over constants) was considered in Example \ref{isocoho}.
\end{remark}

\subsection{Higher difference Galois cohomology}
As we already mentioned in Example \ref{isocoho} and Corollary \ref{fdc}, the first ``difference Galois cohomology'' group was introduced in \cite{BW}. Using the methods which we have  already developed, we introduce in this subsection higher difference Galois cohomology groups.

Let $(G,\sigma_G)$ be a commutative difference group scheme over a difference field $(\ka,s)$. We still assume that $\mathbf{C}(X)=\mathbf{F}(X)$. By Example \ref{mainrex}, $\mathcal{R}(G)$ has a natural structure of a right difference sheaf.
We generalize the notion of the first ``difference Galois cohomology'' from \cite{BW} in the following way.
\begin{definition}\label{hdgc}
For $n\in \Nn$, we define the \emph{$n$-th difference Galois cohomology} of $(\ka,\sigma)$ with coefficients in $(G,\sigma_G)$ as:
$$H^n_{\sigma}(\ka,G):=H^n_{\sigma}\left(\mathbf{F}(\spec(\ka)),\mathcal{R}(G)\right).$$
If $G$ is not necessarily commutative, then we define:
$$H^1_{\sigma}(\ka,G):=\check{H}^1_{\sigma}\left(\mathbf{F}(\spec(\ka)),\mathcal{R}(G)\right).$$
\end{definition}
\begin{remark}
We notice below that the above definition does not contradict itself, and that it generalizes the one from \cite{BW} indeed.
\begin{enumerate}
\item By Theorem \ref{cechtoder2r}, in the case of a commutative $G$, the two definitions above coincide.

\item By Corollary \ref{fdc}, our definition of difference Galois cohomology coincides with the one from \cite{BW} in the case of $n=1$.
\end{enumerate}
\end{remark}
We generalize below (to the case of an arbitrary difference group scheme) the short exact sequence from Example \ref{isocoho}. Similarly as in Section \ref{etaleexam}, the homomorphism $s:\ka\to \ka$ extends to a homomorphism $\tilde{s}:\ka^{\sep}\to \ka^{\sep}$, which, using  \cite[Chapter II.1, Remark 2]{serre2002galois}, induces the following homomorphism not depending on the choice of $\tilde{s}$:
$$H^m(s,G):H^m(\ka,G)\to H^m(\ka,{}^{\sigma}G).$$
Similarly as for the sheaf cohomology, we define the group of ``(co)invariants'' below:
$$H^m(\ka,G)^{\sigma}:=\ker\left(H^m(s,G)-H^m(\ka,\sigma_G)\right),$$
$$H^m(\ka,G)_{\sigma}:=\mathrm{coker}\left(H^m(s,G)-H^m(\ka,\sigma_G)\right).$$
\begin{theorem}\label{wibdesc}
We assume that $G$ is a commutative group scheme over $\ka$, which is smooth and quasi-projective. Then for each $n>0$, we have the following exact sequence:
$$0\ra H^{n-1}(\ka,G)_{\sigma}\ra H^n_{\sigma}(\ka,G)\ra H^n(\ka,G)^{\sigma}\ra 1,$$
where for each $m\in \Nn$, $H^m(\ka,G)$ is the classical Galois cohomology,
\end{theorem}
\begin{proof}
By Definition \ref{hdgc}, we have:
$$H^n_{\sigma}(\ka,G):=H^n_{\sigma}\left(\mathbf{F}(\spec(\ka)),\mathcal{R}(G)\right).$$
Therefore, by Theorem \ref{mainss2}, we get the following exact sequence:
\[0\ra H^{n-1}(\mathbf{F}(\spec(\ka)),\mathcal{R}(G))_{\sigma}\ra H^n_{\sigma}(\ka,G)\ra
H^n(\mathbf{F}(\spec(\ka)),\mathcal{R}(G))^{\sigma}\ra 0,\]
in which the invariants and coinvariants are defined as before the statement of Theorem \ref{mainss2}. By \cite[Theorem III.3.19]{milne1etale}, for each $m\in \Nn$ we have:
$$H^m\left(\mathbf{F}(\spec(\ka)),\mathcal{R}(G)\right)\cong H^m\left(\mathbf{E}(\spec(\ka)),\mathcal{R}(G)\right).$$
By \cite[Example III.1.7(a)]{milne1etale}, for each $m\in \Nn$, we have:
$$H^m\left(\mathbf{E}(\spec(\ka)),\mathcal{R}(G)\right)\cong H^m(\ka,G),$$
which finishes the proof.
\end{proof}
\begin{example}\label{exbw2calc}
We classify here the difference torsors of the difference group scheme $(\ga^n,\sigma_{\ga^n})$ from Example \ref{exbw2}. By Theorem \ref{wibdesc}, we have the following short exact sequence:
$$0\ra \ga^n(\ka)_{\sigma} \ra H^1_{\sigma}\left(\ka,\ga^n\right) \ra H^1(\ka,\ga^n)^{\sigma}\ra 0.$$
By the (additive) Hilbert 90, we get:
$$H^1_{\sigma}(\ka,\ga^n)\cong\ga^n(\ka)_{\sigma}\cong \mathrm{coker}\left(x\mapsto \lambda_0x+\ldots+\lambda_{n-1}\sigma^{n-1}(x)\right).$$
Hence, we immediately see (as it was also shown in \cite[Example 3.8]{BW}) that all the difference $\ga^n$-torsors are of the form $(\ga^n,\sigma_{\lambda})$ for some $\lambda\in \ka$ (see Example \ref{exbw2}), and that for all $\lambda_1,\lambda_2\in \ka$ we have the following:
$$\left(\ga^n,\sigma_{\lambda_1}\right)\cong \left(\ga^n,\sigma_{\lambda_2}\right)\  \ \ \  \ \ \Longleftrightarrow\  \ \ \  \ \ \lambda_1-\lambda_2\in \mathrm{im}\left(x\mapsto \lambda_0x+\ldots+\lambda_{n-1}\sigma^{n-1}(x)\right).$$
\end{example}

\subsection{Difference cohomology and Picard-Vessiot extensions}\label{left1}
In this subsection, we analyze two well-known constructions and interpret some results from \cite{CP} in terms of difference Galois cohomology. Let $G$ be a group scheme over $X$. We can make a difference group scheme out of it in the following two ways. We will both describe the difference group schemes and the corresponding representable functors (see Remark \ref{remdifsch}).
\begin{enumerate}
\item We consider the forgetful functor:
$$\mathbf{Sch}_{(X,\sigma)}\ni (Y,\sigma_Y)\mapsto G(Y)\in \mathbf{Gps}.$$
If the structure morphism $G\to X$ is affine, then (see \cite[Lemma 32.3.1]{stacks}) this functor is represented by the difference group scheme $G^{\infty}$,
where:
$$G^{\infty}:=\prod_{n=0}^{\infty}{}^{\sigma^n}G,\ \ \ \ \sigma_{G^{\infty}}:=\text{left shift}$$
(we are grateful to Michael Wibmer for pointing out to us some issues related with existence of the infinite products in the category of schemes).
The functor $G\mapsto G^{\infty}$ is right adjoint to the forgetful functor from the category of affine difference group schemes to the category of group schemes.

\item If $G$ is defined over constants (see Section \ref{igs}), then we look at the difference group scheme $(G,\id)$. For any difference scheme $(Y,\sigma_Y)$ over $(X,\sigma)$, the group of rational points $G(Y)$ is acted on by $\sigma_Y$ (since $G$ is defined over constants), and the difference group scheme $(G,\id)$ represents the following functor:
$$\mathbf{Sch}_{(X,\sigma)}\ni (Y,\sigma_Y)\mapsto G(Y)^{\sigma_Y}\in \mathbf{Gps}.$$

\end{enumerate}
Let us compute the right difference cohomology groups in each of these cases. We tacitly assume that the group scheme $G$ is commutative, but, in the case of $n=1$, one easily gets the corresponding results for an arbitrary $G$ as well.
\begin{enumerate}
\item By Theorem \ref{mainss2}, we have the following short exact sequence:
	\[0\ra H^{n-1}(X,G^{\infty})_{\sigma}\ra H^n_{\sigma}(X,G^{\infty})\ra H^n(X,G^{\infty})^{\sigma}\ra 0.\]
Since sheaf cohomology commutes with products of the coefficient sheaves, we get that:
$$H^n(X,G^{\infty})\cong \prod_{n=0}^{\infty}H^n\left(X,{}^{\sigma^n}G\right),$$
where the group of ``invariants'' coincides with the equalizer of the left-shift map on the cohomology and the restriction map induced on $\sigma$ on $X$ (similarly with the ``coinvariants'' and the corresponding coequalizer). Therefore, we get:
$$H^n(X,G^{\infty})^{\sigma}\cong H^n(X,G), \ \ \ H^n(X,G^{\infty})_{\sigma}=0,$$
which implies the following (in the special case of $n=1$ as well as both $G$ and $X$ being affine, it is \cite[Proposition 3.5]{BW}):
$$H^n_{\sigma}(X,G^{\infty})\cong H^n(X,G).$$

\item By Theorem \ref{mainss2}, we have the following short exact sequence:
$$0\ra H^{n-1}(X,G)_{\sigma}\ra H^n_{\sigma}(X,G) \ra H^n(X,G)^{\sigma}\ra 0.$$
The kernel part can be understood as the group $\as(H^{n-1}(X,G),\sigma)$ (see Definition \ref{asdef}).


\end{enumerate}
We will use the second construction above to interpret in terms of difference cohomology some results about difference Picard-Vessiot extensions from \cite{CP}. A difference field $(\ka,\sigma)$ is called \emph{strongly PV-closed} (\cite[Definition 5.1]{CP}), if all linear difference equations over $\ka$ have solutions in $(\ka,\sigma)$ (with the extra assumption of $\ka$ being algebraically closed, this notion goes by the name \emph{linearly closed} in \cite[Definition 5.1]{KP4}). We give below an easy cohomological interpretation.
\begin{lemma}\label{pvrem}
The following are equivalent:
\begin{itemize}
\item the difference field $(\ka,\sigma)$ is strongly PV-closed;

\item for all $n>0$, the group $H^1_{\sigma}(\ka,(\gl_n,\id))$ is trivial.
\end{itemize}
\end{lemma}
\begin{proof}
By the short exact sequence from Item $(2)$ above and Hilbert 90, we get:
$$ H^1_{\sigma}(\ka,(\gl_n,\id))\cong \mathrm{AS}\left(\gl_n(\ka),\gl_n(\sigma)\right).$$
It is easy to see that the difference field $(\ka,\sigma)$ is strongly PV-closed if and only if the corresponding Artin-Schreier set is trivial, which finishes the proof.
\end{proof}
In \cite[Prop 5.6]{CP}, there is an example of a strongly PV-closed difference field with non-trivial linear difference torsors (such a situation can not happen in the differential case, see \cite{Pi6}). Using Corollary \ref{fdc} and Remark \ref{pvrem}, one can interpret this result in terms of difference cohomology in the following way.
\begin{theorem}
There is a difference field $(\ka,\sigma)$ such that:
\begin{enumerate}
  \item $H^1_{\sigma}\left(\ka,(\gl_n,\id)\right)=0$, for all $n>0$;
  \item $H^1_{\sigma}\left(\ka,(\gm,x\mapsto x^2)\right) \neq 0$.
\end{enumerate}
\end{theorem}
\begin{proof}
In the proof of \cite[Prop 5.6]{CP}, there is a construction of a strongly PV-closed difference field $(\ka,\sigma)$ and a difference torsor of $(\gm,x\mapsto x^2)$, which has no difference $(\ka,\sigma)$-rational points. By Remark \ref{anotdt}(2), this difference torsor is not trivial, hence we get the result by Corollary \ref{fdc} (and Definition \ref{hdgc}).
\end{proof}

\bibliographystyle{plain}
\bibliography{harvard}

\end{document}